\documentclass[10pt]{article}

\usepackage{amsmath}
\usepackage{amssymb}
\usepackage{graphicx}
\usepackage{epic,eepic}

\usepackage{color}

\usepackage{amstext,amsthm,amssymb,amsmath}
\usepackage{graphicx}
\usepackage{subfigure}
\usepackage{wrapfig}
\usepackage{fullpage}
\usepackage{color}
\usepackage{multirow}
\usepackage{tabulary}
\usepackage{booktabs}
\usepackage{enumerate}

\newtheorem{theorem}{Theorem}[section]
\newtheorem{lemma}[theorem]{Lemma}

\newcommand{\norm}[1]{\left\|#1\right\|}

\hfuzz=\maxdimen
\title{An adaptive GMsFEM for high-contrast flow problems}

\author{Eric T. Chung\thanks{Department of Mathematics, The Chinese University of Hong Kong, Hong Kong SAR. 
This research is partially supported by the Hong Kong RGC General Research Fund (Project number: 400411).},
Yalchin Efendiev\thanks{Department of Mathematics, Texas A\&M University, College Station, TX.}
and Guanliang Li\thanks{Department of Mathematics, Texas A\&M University, College Station, TX.}
}

\begin{document}
\maketitle

\begin{abstract}

In this paper, we derive an a-posteriori error indicator for the Generalized 
Multiscale Finite Element Method (GMsFEM) framework. 
This error indicator is further used
to develop an adaptive enrichment algorithm for
the linear elliptic equation with multiscale high-contrast coefficients.
The GMsFEM, which has recently been 
introduced in \cite{egh12}, allows solving
multiscale parameter-dependent problems at a reduced computational cost
by constructing a reduced-order representation of the solution on a coarse
grid. The main idea of the method consists of (1) the construction
of snapshot space, (2) the construction of the offline space, and (3)
the 
construction of the online space (the latter for parameter-dependent problems).
 In \cite{egh12}, it was shown that the GMsFEM
provides a flexible tool to solve multiscale problems 
with a complex input space by generating
appropriate snapshot, offline, and online spaces. In this paper,
we study an adaptive enrichment procedure and derive an 
a-posteriori error indicator which gives an estimate of the local error 
over coarse grid regions.
We consider two kinds of error indicators where one is based on
the $L^2$-norm of the local residual
and the other is based on the weighted $H^{-1}$-norm of the local residual
where the weight is related to the coefficient of the elliptic equation.
We show that the use of weighted $H^{-1}$-norm residual 
gives a more robust error indicator
which works well for cases with high contrast media. 
The convergence analysis of the method is given.
In our analysis, we do not consider the error due to the fine-grid
discretization of local problems and only study the errors 
due to the enrichment.
Numerical results are presented that demonstrate the robustness
of the proposed error indicators.

\end{abstract}

\section{Introduction}
\label{sec:intro}

Model reduction techniques are often required for solving 
challenging multiscale problems that have multiple scales and high contrast. 
Many of these model
reduction techniques perform the discretization of the problem on a coarse
grid where coarse grid size is much larger than the fine-grid discretization.
The latter requires constructing reduced order models for the solution space
on a coarse grid. Some of these techniques involve upscaled models (e.g., 
\cite{dur91, weh02}) or multiscale methods (e.g., \cite{Arbogast_two_scale_04, Chu_Hou_MathComp_10,ee03,
  egw10,eh09,ehg04, GhommemJCP2013,ReducedCon,MsDG,Wave,WaveGMsFEM}).

In this paper, we derive an a-posteriori error indicator for the Generalized 
Multiscale Finite Element Method (GMsFEM) framework
\cite{egh12}. 
This error indicator is further used
to develop an adaptive enrichment algorithm for
the linear elliptic equation with multiscale high-contrast coefficients.
 GMsFEM is a flexible general
framework that generalizes the Multiscale Finite Element Method (MsFEM)
(\cite{hw97})
by systematically enriching the coarse spaces and taking into account
small scale information and
complex input spaces. This approach, as in many
multiscale model reduction techniques, divides the computation into
two stages: the offline and the online. In the offline stage, 
a small dimensional space is constructed that can be 
used in the online stage to construct multiscale basis functions.
These multiscale basis functions can be re-used for any input parameter
to solve the problem on a coarse grid. The main idea behind the construction
of offline and online spaces is the selection of local spectral problems
and the selection of the snapshot space.
In \cite{egh12}, we propose several general strategies. In this paper,
we investigate adaptive enrichment procedures.

In previous findings \cite{egw10, eglp13}, a-priori error bounds
for the GMsFEM are derived for linear elliptic equations.
It was shown that the convergence rate is proportional to
the inverse of the eigenvalue that corresponds to the eigenvector
which is not included in the coarse space. Thus, adding more basis
functions will improve the accuracy and
it is important to include the eigenvectors that correspond
to very small eigenvalues  (\cite{egw10}). 
Rigorous a-posteriori
error indicators are needed to perform an adaptive enrichment
which is a subject of this paper. We would like to point out that 
there are many related activities in designing a-posteriori
error estimates \cite{ohl12, abdul_yun, dinh13, nguyen13, tonn11} 
 for global reduced models. 
The main difference is that our error estimators are
based on special local eigenvalue problem and use the eigenstructure
of the offline space.

In the paper, we consider two kinds of error indicators where one is based on
the $L^2$-norm of the local residual
and the other is based on the weighted $H^{-1}$-norm (we will also call it 
$H^{-1}$-norm based) of the local residual
where the weight is related to the coefficient of the elliptic equation.
We show that the use of weighted $H^{-1}$-norm residual 
gives a more robust error indicator
which works well for cases with high contrast media. 
The convergence analysis of the method is given.
In our analysis, we do not consider the error due to the fine-grid
discretization of local problems and only study the errors due 
to the enrichment.
In this regard, we assume that the error is largely 
due to coarse-grid discretization.
The fine-grid discretization error can be considered in general
(e.g., as in \cite{abdul_yun, ohl12}) and 
this will give an additional error estimator.
The proposed error indicators allow adding multiscale
basis functions in the regions detected by the error indicator.
The multiscale basis functions are  selected by choosing
next important eigenvectors (based on the increase of the eigenvalues)
 from the offline space.

The convergence proof of our adaptive enrichment algorithm is based on 
the techniques used for proving the convergence of adaptive refinement method
for classical conforming finite element methods for second order elliptic problems
\cite{BrennerScott,AdaptiveFEM}.
Contrary to \cite{AdaptiveFEM} where mesh refinement is considered,
we prove the convergence of our adaptive enrichment algorithm
as the approximation space is enriched for a fixed coarse mesh size.
The convergence is based on some
previously developed
spectral estimates. In particular, we use both stability
of the coarse-grid projection and the convergence of spectral interpolation.
Another key idea is that our error indicators are defined in a variational sense
instead of the pointwise residual of the differential equation.
By using this variational definition, 
we avoid the use of the gradient of the multiscale coefficient.
Moreover, our convergence analysis does 
not require that the gradient of the coefficient is bounded,
which is not the case for high-contrast multiscale flow problems.

In the proposed error indicators, we consider
the use of snapshot space in GMsFEM. In this case,
the residual contains an irreducible error due
to the difference between the snapshot solution
and the fine-grid solution. We consider the use of 
snapshot space for approximating the residual error
in the case of 
 weighted $H^{-1}$-norm of the local residual.

We present several numerical tests by considering
two different high-contrast multiscale permeability fields.
We study both 
 error indicators based on
the $L^2$-norm of the local residual
and  the weighted $H^{-1}$-norm of the local residual. 
Our numerical results show that 
 the use of weighted $H^{-1}$-norm residual 
gives a more robust error indicator
which works well for cases with high contrast media. 
In our numerical results, we also compare the results obtained
by the proposed
indicators and the exact error indicator which is computed
by considering the energy norm of the difference between
the fine-scale solution and the offline solution.
Our numerical results show that the use of the exact error
indicator gives nearly similar results to the case of using
weighted $H^{-1}$ error indicator.
In our studies, we also consider
the errors between the fine-grid solution and the offline
solution as well as the snapshot solution and the offline solution.
All cases show that the proposed  error indicator
is robust and can be used to detect 
the regions where additional multiscale basis 
functions are needed.

The paper is organized in the following way. 
In the next section, we present Preliminaries. 
The GMsFEM is presented in Section \ref{cgdgmsfem}.
In Section 
\ref{sec:errorindicator},
we present the details of the error indicator and state our main
results. In Section \ref{sec:numresults}, numerical results are presented.
The proofs of our main results are presented in Section \ref{sec:analysis}.
The paper ends with a Conclusion.

\section{Preliminaries}
\label{prelim}
In this paper, we consider high-contrast flow problems of the form
\begin{equation} \label{eq:original}
-\mbox{div} \big( \kappa(x) \, \nabla u  \big)=f \quad \text{in} \quad D,
\end{equation}
subject to the homogeneous Dirichlet boundary condition $u=0$ on $\partial D$,
where $D$ is the computational domain. 
We assume that $\kappa(x)$ is a heterogeneous coefficient
with multiple scales and very high contrast. 
To discretize (\ref{eq:original}), we 
introduce the notion of fine and coarse grids. 
We let $\mathcal{T}^H$ be a usual conforming partition of the computational domain $D$ into finite elements (triangles, quadrilaterals, tetrahedra, etc.). We refer to this partition as the coarse grid and assume that each coarse element is partitioned into a connected union of fine grid blocks. The fine grid partition will be denoted by $\mathcal{T}^h$, and is by definition a refinement of the coarse grid $\mathcal{T}^H$. 
We use $\{x_i\}_{i=1}^{N}$ (where $N$ denotes the number of coarse nodes) to denote the vertices of
the coarse mesh $\mathcal{T}^H$, and define the neighborhood of the node $x_i$ by
\begin{equation} \label{neighborhood}
\omega_i=\bigcup\{ K_j\in\mathcal{T}^H; ~~~ x_i\in \overline{K}_j\}.
\end{equation}
See Figure~\ref{schematic} for an illustration of neighborhoods and elements subordinated to the coarse discretization. We emphasize the use of $\omega_i$ to denote a coarse neighborhood, and $K$ to denote a coarse element throughout the paper.

\begin{figure}[htb]
  \centering
  \includegraphics[width=0.65 \textwidth]{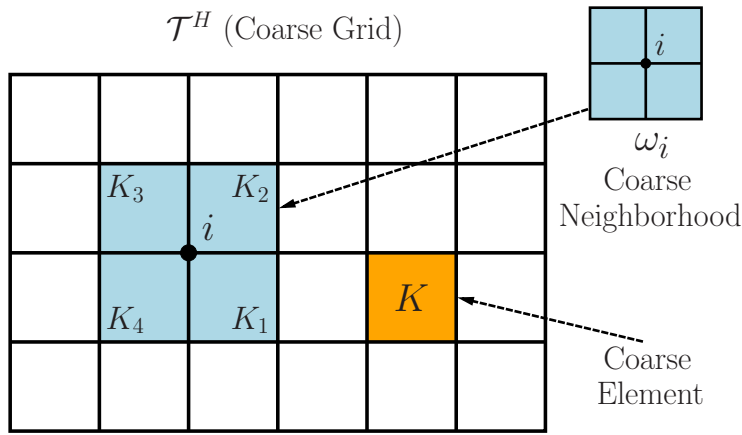}
  \caption{Illustration of a coarse neighborhood and coarse element}
  \label{schematic}
\end{figure}

Next, we briefly outline the GMsFEM.
We will consider the 
continuous Galerkin (CG) formulation and we will use $\omega_i$ as the support of basis functions. 
%
We 
denote the basis functions by $\psi_k^{\omega_i}$, which is supported in $\omega_i$.
In particular, we note that the proposed approach will employ the use of multiple basis functions per coarse neighborhood,
and the index $k$ represents the numbering of these basis functions.
In turn, the CG solution will be sought as $u_{\text{ms}}(x)=\sum_{i,k} c_{k}^i \psi_{k}^{\omega_i}(x)$.
Once the basis functions are identified, the CG global coupling is given through the variational form
\begin{equation}
\label{eq:globalG} a(u_{\text{ms}},v)=(f,v), \quad \text{for all} \, \, v\in
V_{\text{off}},
\end{equation}
where  $V_{\text{off}}$ is used to denote the space spanned by those basis functions
and $a(\cdot,\cdot)$ is a usual bilinear form corresponding to (\ref{eq:original}).
We also note that one can use 
discontinuous Galerkin formulation (see e.g., \cite{Wave,WaveGMsFEM,eglmsMSDG}) to couple multiscale basis functions
defined on $K$.

Let $V$ be the conforming finite element space
with respect to the fine-scale partition $\mathcal{T}^h$. 
We assume $u\in V$ is the fine-scale solution satisfying
\begin{equation*}
a(u,v) = (f,v), \quad v\in V.
\end{equation*}

\section{CG-based GMsFEM for high-contrast flow problems}
\label{cgdgmsfem}

In this section, we will give a brief description
of the GMsFEM for high contrast flow problems. 
More details can be found in \cite{egh12, eglp13}.
In the following, we also give a general outline of the GMsFEM. 

\begin{itemize}
\item[1.]  Offline computations:
\begin{itemize}
\item 1.0. Coarse grid generation.
\item 1.1. Construction of snapshot space that will be used to compute an offline space.
\item 1.2. Construction of a small dimensional offline space by performing dimension reduction in the space of global snapshots.
\end{itemize}
\item[2.] Online computations:
\begin{itemize}
\item 2.1. For each input parameter, compute multiscale basis functions.  (for parameter-dependent cases)
\item 2.2. Solution of a coarse-grid problem for any force term and boundary condition.
\item 2.3. Iterative solvers, if needed.
\end{itemize}
\end{itemize}

\subsection{Local basis functions}
\label{locbasis}

We now present the construction
of the basis functions 
and the corresponding spectral problems 
for obtaining a space reduction. 


In the offline computation, we first construct a snapshot space $V_{\text{snap}}^{\omega}$. 
The snapshot space can be the space of all fine-scale basis functions
or the solutions of some local problems with various choices of boundary conditions. 
For example, we can use the following $\kappa$-harmonic extensions to form a snapshot space.
 For each fine-grid function, $\delta_j^h(x)$,
which is defined by 
$\delta_j^h(x)=\delta_{j,k},\,\forall j,k\in \textsl{J}_{h}(\omega_i)$, where $\textsl{J}_{h}(\omega_i)$ denotes the fine-grid boundary node on $\partial\omega_i$. \\

Given a fine-scale piecewise linear function defined on $\partial\omega$, we define $\psi_{j}^{\omega_i, \text{snap}}$ by
\begin{equation} \label{harmonic_ex}
-\text{div}(\kappa(x) \nabla \psi_{j}^{\omega_i, \text{snap}} ) = 0
\quad \text{in} \, \, \, \omega_i,
\end{equation}
where $\psi_{j}^{\omega_i, \text{snap}}=\delta_j^h(x)$ on $\partial\omega_i$.\\
For brevity of notation we now omit the superscript $\omega_i$, yet it is assumed throughout this section that the offline space computations are localized to respective coarse subdomains. 
Let $W_i$ be the number of functions in the snapshot space in the region $\omega_i$, and
$$
V_{\text{snap}} = \text{span}\{ \psi_{j}^{ \text{snap}}:~~~ 1\leq l \leq W_i \},
$$
for each coarse subdomain $\omega_i$. 

Denote
$$
R_{\text{snap}} = \left[ \psi_{1}^{\text{snap}}, \ldots, \psi_{W_i}^{\text{snap}} \right].
$$

In order to construct the offline space $V_{\text{off}}^\omega$, we perform a dimension reduction of the space of snapshots using an auxiliary spectral decomposition. The analysis in \cite{egw10} motivates the following eigenvalue problem in the space of snapshots:
\begin{equation} \label{offeig}
A^{\text{off}} \Psi_k^{\text{off}} = \lambda_k^{\text{off}} S^{\text{off}} \Psi_k^{\text{off}},
\end{equation}
where
\begin{equation*}
 \displaystyle A^{\text{off}} = [a_{mn}^{\text{off}}] = \int_{\omega} \kappa(x) \nabla \psi_m^{\text{snap}} \cdot \nabla \psi_n^{\text{snap}} = R_{\text{snap}}^T A R_{\text{snap}}
 \end{equation*}
 \begin{center}
 and
 \end{center}
 \begin{equation*}
 \displaystyle S^{\text{off}} = [s_{mn}^{\text{off}}] = \int_\omega  \widetilde{\kappa}(x)\psi_m^{\text{snap}} \psi_n^{\text{snap}} = R_{\text{snap}}^T S R_{\text{snap}},
\end{equation*}
where $A$ and $S$ denote analogous fine scale matrices as defined by
\begin{equation*}
A_{ij} = \int_{D} \kappa(x) \nabla \phi_i \cdot \nabla \phi_j \, dx
\quad
S_{ij} = \int_{D} \widetilde{\kappa}(x)  \phi_i  \phi_j \, dx
\end{equation*} 
where $\phi_i$ is the fine-scale basis function. We will give the definition
of $\widetilde{\kappa}(x)$ later on.
To generate the offline space we then choose the smallest $M^{\omega}_{\text{off}}$ eigenvalues from Eq.~\eqref{offeig} and form the corresponding eigenvectors in the space of snapshots by setting
$\psi_k^{\text{off}} = \sum_j \Psi_{kj}^{\text{off}} \psi_j^{\text{snap}}$ (for $k=1,\ldots, M^{\omega}_{\text{off}}$), where $\Psi_{kj}^{\text{off}}$ are the coordinates of the vector $\Psi_{k}^{\text{off}}$. 

\subsection{Global coupling}
\label{globcoupling}

In this section we  create an appropriate solution space and variational formulation that for a continuous Galerkin approximation of Eq.~\eqref{eq:original}. We begin with an initial coarse space $V^{\text{init}}_0 = \text{span}\{ \chi_i \}_{i=1}^{N}$, where the $\chi_i$ are the standard multiscale partition of unity functions defined by
\begin{eqnarray} \label{pou}
-\text{div} \left( \kappa(x) \, \nabla \chi_i  \right) = 0 \quad K \in \omega_i \\
\chi_i = g_i \quad \text{on} \, \, \, \partial K, \nonumber
\end{eqnarray}
for all $K \in \omega_i$, where $g_i$ is assumed to be linear. We note that the summed, pointwise energy $\widetilde{\kappa}$ required for the eigenvalue problems will be defined as
\begin{equation*}
\widetilde{\kappa} = \kappa \sum_{i=1}^{N} H^2 | \nabla \chi_i |^2,
\end{equation*}
where $H$ denotes the coarse mesh size.

We then multiply the partition of unity functions by the eigenfunctions in the offline space $V_{\text{off}}^{\omega_i}$ to construct the resulting basis functions
\begin{equation} \label{cgbasis}
\psi_{i,k} = \chi_i \psi_k^{\omega_i, \text{off}} \quad \text{for} \, \, \,
1 \leq i \leq N \, \, \,  \text{and} \, \, \, 1 \leq k \leq M_{\text{off}}^{\omega_i},
\end{equation}
where $M_{\text{off}}^{\omega_i}$ denotes the number of offline eigenvectors that are chosen for each coarse node $i$. We note that the construction in Eq.~\eqref{cgbasis} yields  continuous basis functions due to the multiplication of offline eigenvectors with the initial (continuous) partition of unity. Next, we define the continuous Galerkin spectral multiscale space as
\begin{equation} \label{cgspace}
V_{\text{off}}  = \text{span} \{ \psi_{i,k} : \,  \, 1 \leq i \leq N \, \, \,  \text{and} \, \, \, 1 \leq k \leq M_{\text{off}}^{\omega_i}  \}.
\end{equation}
Using a single index notation, we may write $V_{\text{off}} = \text{span} \{ \psi_{i} \}_{i=1}^{N_c}$, where $N_c$ denotes the total number of basis functions that are used in the coarse space construction. We also construct an operator matrix $R_0^T = \left[ \psi_1 , \ldots, \psi_{N_c} \right]$ (where $\psi_i$ are used to denote the nodal values of each basis function defined on the fine grid), for later use in this subsection.

We seek $u_{\text{ms}}(x) = \sum_i c_i \psi_i(x) \in V_{\text{off}}$ such that
\begin{equation} \label{cgvarform}
a(u_{\text{ms}}, v) = (f, v) \quad \text{for all} \,\,\, v \in V_{\text{off}},
\end{equation}
where
$ \displaystyle a(u, v) = \int_D \kappa(x) \nabla u \cdot \nabla v \, dx$, and $ \displaystyle (f,v) = \int_D f v \, dx$. We note that variational form in \eqref{cgvarform} yields the following linear algebraic system
\begin{equation}
A_0 U_0 = F_0,
\end{equation}
where $U_0$ denotes the nodal values of the discrete CG solution, and
\begin{equation*}
A_0 = [a_{IJ}] = \int_D \kappa(x)  \, \nabla \psi_I \cdot \nabla \psi_J \, dx \quad \text{and} \quad F_0 = [f_I] = \int_D f \psi_I \, dx.
\end{equation*}
Using the operator matrix $R_0^T$, we may write $A_0 = R_0 A R_0^T$ and $F_0 = R_0 F$, where $A$ and $F$ are the standard, fine scale stiffness matrix and forcing vector corresponding to the form in Eq.~\eqref{cgvarform}. We also note that the operator matrix may be analogously used in order to project coarse scale solutions onto the fine grid.




\section{A-posteriori error estimate and adaptive enrichment}
\label{sec:errorindicator}

In this section, we will derive an a-posteriori error indicator
for the error $u-u_{\text{ms}}$ in energy norm.
We will then use the error indicator 
to develop an adaptive enrichment algorithm.
The a-posteriori error indicator
gives an estimate of the local error on the coarse grid regions $\omega_i$,
and we can then add basis functions to improve the solution. 
We will give two kinds of error indicators,
one is based on the $L^2$-norm of the local residual
and the other is based on the weighted
 $H^{-1}$-norm of the local residual (for simplicity, we will also call it
$H^{-1}$-norm based indicator). 
The $L^2$-norm residual is also used in the classical adaptive finite element method.
In our case, this type of error indicator works well when the coefficient 
does not contain high contrast region. 
We will provide a quantitative explanation for this in the next section. 
On the other hand, the $H^{-1}$-norm based residual 
gives a more robust error indicator
which works well for cases with high contrast media. 
This section is devoted to the derivation of the a-posteriori error indicator
and the corresponding adaptive enrichment algorithm.
The convergence analysis of the method will be given in the next section.

Let $V$ be the fine scale finite element space. We recall that the fine scale solution $u$ satisfies
\begin{equation}
a(u, v) = (f, v) \quad \text{for all} \,\,\, v \in V
\label{eq:fine}
\end{equation}
and the multiscale solution $u_{\text{ms}}$ satisfies
\begin{equation}
a(u_{\text{ms}}, v) = (f, v) \quad \text{for all} \,\,\, v \in V_{\text{off}}.
\label{eq:coarse}
\end{equation}
We remark that 
$V_\text{off} \subset V$. 
Next we will give the definitions of the $L^2$-based and $H^{-1}$-based residuals. 

{\bf $L^2$-based residual}:

Let $\omega_i$ be a coarse grid region. We define a linear functional $Q_i(v)$ on $L^2(\omega_i)$ by
\begin{equation}
Q_i(v) =  \int_{\omega_i} fv\chi_i - \int_{\omega_i} a\nabla u_{\text{ms}}\cdot \nabla (v\chi_i).
\end{equation}
The norm of $Q_i$ is defined as
\begin{equation}
\| Q_i \| = \sup_{v\in L^2(\omega_i)} \frac{ |Q_i(v)| }{\| v\|_{L^2(\omega_i)}}.
\end{equation}
The norm $\|Q_i\|$ gives an estimate on the size of error. 

{\bf $H^{-1}$-based residual}:

Let $\omega_i$ be a coarse grid region and let $V_i = H^1_0(\omega_i)$.
We define a linear functional $R_i(v)$ on $V_i$ by
\begin{equation}
R_i(v) =  \int_{\omega_i} fv - \int_{\omega_i} a\nabla u_{\text{ms}}\cdot \nabla v.
\end{equation}
The norm of $R_i$ is defined as
\begin{equation}
\| R_i \|_{V_i^*} = \sup_{v\in V_i} \frac{ |R_i(v)| }{\| v\|_{V_i}}.
\end{equation}
where $\|v\|_{V_i} = a(v,v)^{\frac{1}{2}}$. 
The norm $\|R_i\|_{V_i^*}$ gives an estimate on the size of error. 

To simplify notations, we let $l_i = M_{\text{off}}^{\omega_i}$. 
We recall that, for each $\omega_i$, 
the eigenfunctions corresponding to $\lambda_1^{\omega_i}, \cdots, \lambda_{l_i}^{\omega_i}$
are used in the construction of $V_{\text{off}}$. 
We also define $\widetilde{\kappa}_i = \min_{x\in \omega_i} \widetilde{\kappa}(x)$. 

In the next section, we will prove the following theorem.
\begin{theorem}
Let $u$ and $u_{\text{ms}}$ be the solutions of \eqref{eq:fine} and \eqref{eq:coarse} respectively. Then
\begin{eqnarray}
\| u-u_{\text{ms}}\|_V^2 &\leq& C_{\text{err}}\sum_{i=1}^N \|Q_i\|^2  ( \widetilde{\kappa}_i \lambda^{\omega_i}_{l_i+1})^{-1}, \label{eq:res1} \\
\| u-u_{\text{ms}}\|_V^2 &\leq& C_{\text{err}}  \sum_{i=1}^N \|R_i\|^2_{V_i^*}  (\lambda^{\omega_i}_{l_i+1})^{-1}. \label{eq:res2}
\end{eqnarray}
\label{thm:post}
where $C_{\text{err}}$ is a uniform constant. 
\end{theorem}

From \eqref{eq:res1} and \eqref{eq:res2}, we see that 
the norms $\|Q_i\|$ and $\|R_i\|_{V_i^*}$
give indications on the size of the energy norm error $\| u-u_{\text{ms}}\|_V$. 
Even though \eqref{eq:res1} and \eqref{eq:res2}
have the same form, we emphasize that
they give different convergence behavior
in the high contrast case.

We will now present the adaptive enrichment algorithm. 
We use $m \geq 1$ to represent the enrichment level
and $V_{\text{off}}^m$ be the solution space at level $m$. 
For each coarse region, we use $l_i^m$
be the number of eigenfunctions used at the enrichment level $m$
for the coarse region $\omega_i$.

{\bf Adaptive enrichment algorithm}: Choose $0 < \theta < 1$. 
For each $m=1,2, \cdots$, 

\begin{enumerate}
\item[Step 1:] Find the solution in the current space. That is, 
find $u_{\text{ms}}^m \in V^m_{\text{off}}$ such that 
\begin{equation}
a(u^m_{\text{ms}}, v) = (f, v) \quad \text{for all} \,\,\, v \in V^m_{\text{off}}.
\label{eq:solve}
\end{equation}

\item[Step 2:] Compute the local residual. For each coarse region $\omega_i$, we compute
\begin{equation*}
\eta^2_i =
\begin{cases} 
& \|Q_i\|^2  (\widetilde{\kappa}_i \lambda^{\omega_i}_{l^m_i+1})^{-1},\quad \text{ for } L^2\text{-based residual} \\
& \|R_i\|^2_{V_i^*}  (\lambda^{\omega_i}_{l^m_i+1})^{-1},\quad \text{ for } H^{-1}\text{-based residual}
\end{cases}
\end{equation*}
And we re-enumerate them in the decreasing order, that is, $\eta^2_1 \geq \eta^2_2 \geq \cdots \geq \eta^2_N$. 

\item[Step 3:] Find the coarse region where enrichment is needed. We choose the smallest integer $k$ such that
\begin{equation}
\theta \sum_{i=1}^N \eta_i^2 \leq \sum_{i=1}^k \eta_i^2. 
\label{eq:criteria}
\end{equation}

\item[Step 4:] Enrich the space. For each $i=1,2,\cdots, k$, we add basis function
for the region $\omega_i$ according to the following rule.
Let $s$ be the smallest positive integer such that 
$\lambda_{l_i^m+s+1}$ is large enough (see the proof of Theorem \ref{thm:conv}) compared with $\lambda_{l_i^m+1}$. 
Then
we include
the eigenfunctions $\Psi^{\text{off}}_{l_i^m+1}, \cdots, \Psi^{\text{off}}_{l_i^m+s}$
in the construction of the basis functions. 
The resulting space is denoted as $V_{\text{off}}^{m+1}$. 

\end{enumerate}

We remark that the choice of $s$ above will ensure the convergence of the enrichment algorithm,
and in practice, the value of $s$ is easy to obtain. 
Moreover, contrary to classical adaptive refinement methods, 
the total number of basis functions that we can add 
is bounded by the dimension of the snapshot space. 
Thus, the condition \eqref{eq:criteria} can be modified as follows. 
We choose the smallest integer $k$ such that 
\begin{equation*}
\theta \sum_{i=1}^N \eta_i^2 \leq \sum_{i\in I} \eta_i^2
\end{equation*}
where the index set $I$ is a subset of $\{ 1,2, \cdots, k\}$
and contains indices $j$ such that $l_j^m$
is less than the maximum number of eigenfunctions for the region $\omega_j$.

We now describe how the norms $\|Q_i\|$ and $\|R_i\|_{V_i^*}$ are computed. 
Let $W_i$ be the diagonal matrix containing the nodal values of the fine grid cut-off function $\chi_i$ in the diagonal. Then
the norm $\|Q_i\|$ can be computed as
\begin{equation}
\label{eq:normQ}
\|Q_i\| = \| W_i AR_0^T U_0 \|. 
\end{equation}
According to the Riez representation theorem, the norm $\|R_i\|_{V_i^*}$ can be computed as follows.
Let $z_i$ be the solution of 
\begin{equation}\label{eq:loc_Dirichlet}
\int_{\omega_i} a\nabla z_i \cdot \nabla v = R_i(v), \quad \text{for all} \,\,\, v \in V_i.
\end{equation}
Then we have $\|R_i\|_{V_i^*} = \| z_i \|_{V_i}$.
Thus, to find the norm $\|R_i\|_{V_i^*}$,
we need to solve a local problem on each coarse region $\omega_i$.

Finally, we state the convergence theorem.
\begin{theorem}
There are positive constants $\tau, \delta, \rho, L_1$ and $L_2$ such that the following contracting property holds
\begin{equation*}
\| u-u_{\text{ms}}^{m+1}\|_V^2 + \frac{\tau}{1+\tau \delta L_2} \sum_{i=1}^N S_{m+1}(\omega_i)^2
\leq \varepsilon \Big( \|u-u_{\text{ms}}^m\|_V^2 + \frac{\tau}{1+\tau\delta  L_2}  \sum_{i=1}^N S_{m}(\omega_i)^2 \Big).
\end{equation*}
\label{thm:conv}
Note that $0 < \varepsilon < 1$ and 
\begin{equation*}
\varepsilon = \max( 1- \frac{\theta^2}{L_1(1+\tau \delta L_2)}, \frac{2C_{\text{err}}}{\tau L_1}+\rho).
\end{equation*}
\end{theorem}

We remark that the precise definitions of the constants 
$\tau, \delta, \rho, L_1$ and $L_2$
are given in Section \ref{sec:analysis}.

\section{Numerical Results}
\label{sec:numresults}
%
In this section, we will present some numerical experiments to show
the performance of the error indicators and the adaptive enrichment algorithm. 
We take the domain $\Omega$ as a square,
set the forcing term $f=1$ and use a linear boundary condition for the problem \eqref{eq:original}.
In our numerical simulations, we use a $20 \times 20$ coarse grid,
and each coarse grid block is divided into $5\times 5$ fine grid blocks.
Thus, the whole computational domain is partitioned by a $100 \times 100$ fine grid.
We assume that the fine-scale solution is obtained
by discretizing \eqref{eq:original} by the classical conforming piecewise bilinear elements on the fine grid.
To test the performance of our algorithm, we consider two 
permeability fields $\kappa$ as depicted in Figure \ref{fig:perms}. 
We obtain similar numerical results for these cases, and therefore
we will only demonstrate the numerical results for the first permeability field (Figure \ref{fig:perm_HCC}).

Below, we list the indicators used in our simulations. 
In particular, we will recall the definitions of the $L^2$-based and $H^{-1}$-based error indicators.
For comparison purpose, we also use an indicator computed by the exact error in energy norm.
We remark that the indicators are computed for each coarse neighborhood $\omega_i$ and are defined as follows.
\begin{itemize}
\item
The indicator constructed using the weighted $H^{-1}$-based residual is
\begin{equation}\label{eq:indicator_Numerical}
\eta^{\mbox{\scriptsize{En}}}_{\omega_i}= \|R_i\|^2_{V_i^*} (\lambda_{l^m_i+1}^{\omega_i})^{-1} 
\end{equation}
and we name it the proposed indicator.
\item
The indicator constructed using the $L^{2}$-based residual is
\begin{equation}\label{eq:indicator_Numerical_l2}
\eta^{\mbox{\scriptsize{L2}}}_{\omega_i}= \|Q_i\|^2  (\widetilde{\kappa}_i \lambda^{\omega_i}_{l^m_i+1})^{-1}
\end{equation}
and we name it the $L^2$ indicator.
\item
The indicator constructed using the exact energy error is
%
\begin{align}\label{eq:indicator_cmp}
\eta^{\mbox{\scriptsize{Ex}}}_{\omega_i}=\norm{u-u_{\text{ms}}}_{V_i}^{2}
\end{align}
and name it the exact indicator.
\end{itemize}

We recall that, in the above definitions,  the norms $\|Q_i\|$ and $\|R_i\|_{V_i^*}$
are computed in the way described in \eqref{eq:normQ} and \eqref{eq:loc_Dirichlet} respectively. 
For each enrichment level, we will 
compute the multiscale solution (Step 1) and the corresponding error indicators (Step 2). 
The indicators 
$\eta^{\mbox{\scriptsize{Ex}}}_{\omega_i}$, $\eta^{\mbox{\scriptsize{En}}}_{\omega_i}$ and $\eta^{\mbox{\scriptsize{L2}}}_{\omega_i}$ 
are then ordered in decreasing order. 
To enrich the approximation space, 
we select a few coarse neighborhoods such that (\ref{eq:criteria}) holds for a specific value of $\theta$ (Step 3).
In our simulations, we consider $\theta=0.7$ and $0.2$. 
Finally, for selected coarse neighborhoods, we will enrich the offline space
by adding more basis functions (Step 4).


%
%
\begin{figure}[htb]
\centering
 \subfigure[Permeability field 1]{\label{fig:perm_HCC}
    \includegraphics[width = 0.40\textwidth, keepaspectratio = true]{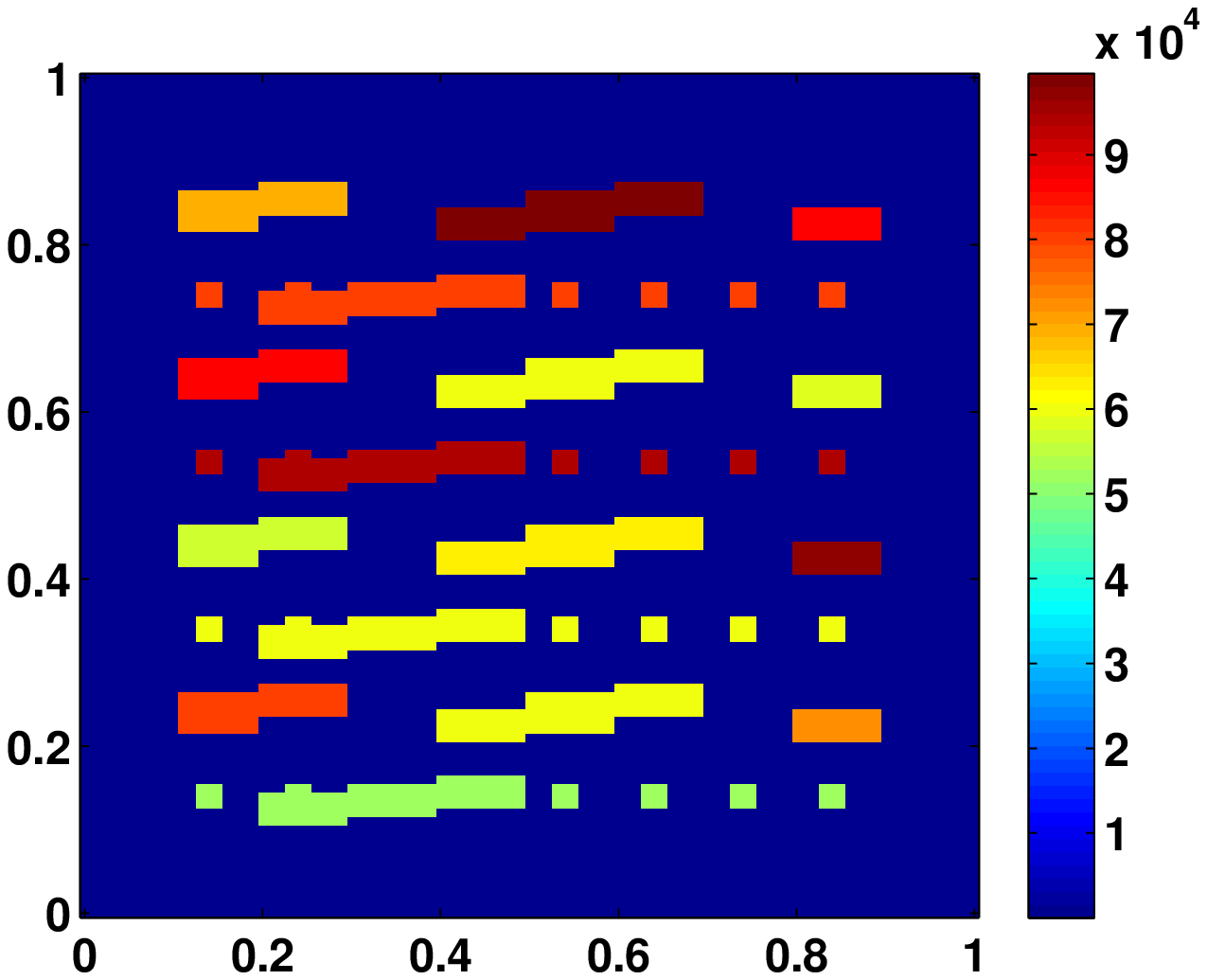}
   }
  \subfigure[Permeability field 2]{\label{fig:newperm2_cross2}
    \includegraphics[width = 0.40\textwidth, keepaspectratio = true]{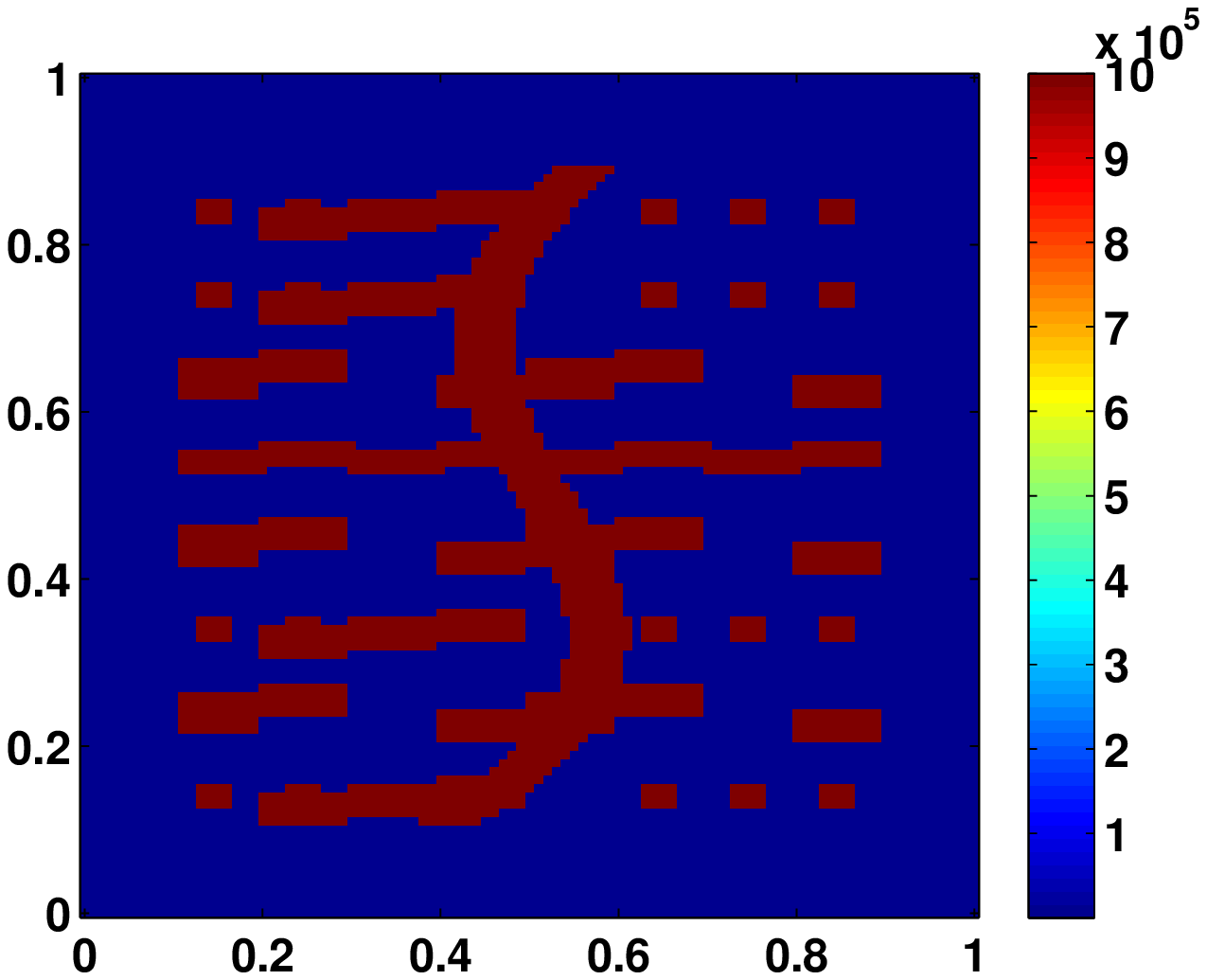}
   }
 \caption{Permeability fields}
 \label{fig:perms}
\end{figure}
%
%

We will consider two types of snapshot spaces,
namely the space spanned by all $\kappa$-harmonic extensions and
the space spanned by all fine-scale conforming piecewise bilinear functions. 
The sequence of offline basis functions is then obtained by
solving the local spectral problem \eqref{offeig}
on the space of snapshots. 
We will call the first type of basis functions as harmonic basis 
and the second type of basis functions as spectral basis. 
In addition, we use the notations $\eta^{\mbox{\scriptsize{H,En}}}_{\omega_i}$, $\eta^{\mbox{\scriptsize{H,L2}}}_{\omega_i}$
and $\eta^{\mbox{\scriptsize{H,Ex}}}_{\omega_i}$
to denote the $H^{-1}$-based, $L^2$-based and exact error indicators
for the case when the offline space is formed by harmonic basis. 
Similarly, we use the notations $\eta^{\mbox{\scriptsize{U,En}}}_{\omega_i}$, $\eta^{\mbox{\scriptsize{U,L2}}}_{\omega_i}$
and $\eta^{\mbox{\scriptsize{U,Ex}}}_{\omega_i}$
to denote the $H^{-1}$-based, $L^2$-based and exact error indicators
for the case when the offline space is formed by spectral basis
(here, superscript $U$ stands for the fact that the snapshot space 
consists of all fine-grid {\it unit} vectors).


%
In the following, we summarize the numerical examples we considered in this paper.
\begin{itemize}
\item {\bf Numerical results with harmonic basis (see Section \ref{sec:621}).}
We will present
 numerical results to test the performance of the error indicator
 $\eta^{\mbox{\scriptsize{H, En}}}_{\omega_i}$ and the adaptive enrichment algorithm 
 with $\theta=0.7$ and $\theta=0.2$. We also compare our results with the use of $\eta^{\mbox{\scriptsize{H,Ex}}}_{\omega_i}$ 
with $\theta=0.7$.
\item {\bf Numerical results with spectral basis (see Section \ref{sec:622}).}
We will present
 numerical results to test the performance of the error indicator
 $\eta^{\mbox{\scriptsize{U, En}}}_{\omega_i}$ and the adaptive enrichment algorithm
 with $\theta=0.7$ and $\theta=0.2$. We also compare our results with the use of $\eta^{\mbox{\scriptsize{U,Ex}}}_{\omega_i}$ 
with $\theta=0.7$.
\item {\bf Numerical results with $L^2$ indicator (see Section \ref{sec:623}).}
We will present
 numerical results to test the performance of the error indicator
  $\eta^{\mbox{\scriptsize{H,L2}}}_{\omega_i}$ and the adaptive enrichment algorithm with $\theta=0.7$.
\item {\bf Numerical results when the proposed indicator is computed in the snapshot space (see Section \ref{sec:624}).}
We will present
 numerical results to test the performance of the error indicator
$\eta^{\mbox{\scriptsize{H,En}}}_{\omega_i}$ and the adaptive enrichment algorithm with $\theta=0.7$.
In this case, the norm $\|R_i\|_{V_i^*}$ is computed in the snapshot space instead of the fine-grid space.
\end{itemize}

In the following,
we will give a brief summary of our conclusions before discussing the numerical results.
\begin{itemize}
\item The use of both $\eta^{\mbox{\scriptsize{H, En}}}_{\omega_i}$
and $\eta^{\mbox{\scriptsize{U, En}}}_{\omega_i}$
gives a convergent sequence of numerical solutions. This verfies
the convergence of our adaptive GMsFEM.
%
\item 
The performance of the proposed indicators $\eta^{\mbox{\scriptsize{H, En}}}_{\omega_i}$
and $\eta^{\mbox{\scriptsize{U, En}}}_{\omega_i}$
is similar to that of the exact indicators $\eta^{\mbox{\scriptsize{H,Ex}}}_{\omega_i}$ and $\eta^{\mbox{\scriptsize{U,Ex}}}_{\omega_i}$.
Thus, the proposed indicator gives a good estimate of the exact error.

%
\item The performance of the weighted $H^{-1}$-based indicator
is much better than that of the $L^2$-based indicator for high-contrast
problems.
\item 
The use of the snapshot space to compute $\eta^{\mbox{\scriptsize{H, En}}}_{\omega_i}$
and $\eta^{\mbox{\scriptsize{U, En}}}_{\omega_i}$
in \eqref{eq:loc_Dirichlet} gives similar results 
compared to the use of local fine-grid solves.
Thus, the computations of $\eta^{\mbox{\scriptsize{H, En}}}_{\omega_i}$
and $\eta^{\mbox{\scriptsize{U, En}}}_{\omega_i}$
can be performed efficiently. 
%
\item With the use of $\theta=0.2$, we obtain more accurate results for the same
dimensional offline spaces compared with $\theta=0.7$.
\end{itemize}
In the tables listed below, we recall that
$V_{\text{off}}$ denotes the offline space; $u$, $u_{\text{snap}}$ and $u_{\text{off}}$ denote the fine-scale, snapshot and offline solutions respectively.
Moreover, to compare the results, we will
compute the error $u-u_{\text{off}}$ using
the $L^2$ relative error and the energy relative error, which are defined as
\begin{equation}
\| u-u_{\text{off}}\|_{L^2_{\kappa}(D)} := \frac{ \| u - u_{\text{off}} \|_{L^2(V)} }{ \| u \|_{L^2(V)} }, \quad\quad
\| u-u_{\text{off}}\|_{H^1_{\kappa}(D)} := \frac{ a(u-u_{\text{off}}, u-u_{\text{off}})^{\frac{1}{2}} }{ a(u,u)^{\frac{1}{2}} }
\end{equation}
where the weighted $L^2$-norm is defined as $\|u\|_{L^2(V)} = \| \widetilde{\kappa}^{\frac{1}{2}} u \|_{L^2(D)}$. 
We will also compute the error $u_{\text{snap}}-u_{\text{off}}$ using the same norms
\begin{equation}
\| u_{\text{snap}}-u_{\text{off}}\|_{L^2_{\kappa}(D)} := \frac{ \| u_{\text{snap}} - u_{\text{off}} \|_{L^2(V)} }{ \| u_{\text{snap}} \|_{L^2(V)} }, \quad\quad
\| u_{\text{snap}}-u_{\text{off}}\|_{H^1_{\kappa}(D)} := \frac{ a(u_{\text{snap}}-u_{\text{off}}, u-u_{\text{off}})^{\frac{1}{2}} }{ a(u_{\text{snap}},u_{\text{snap}})^{\frac{1}{2}} }.
\end{equation}
%
%
%
\subsection{Numerical results with harmonic basis}
\label{sec:621}
In this section, we present numerical examples to test the performance of the proposed indicator $\eta^{\mbox{\scriptsize{H,En}}}_{\omega_i}$ 
and the convergence of our adaptive enrichment algorithm with $\theta=0.7$ and $\theta=0.2$.
We will also compare our results with the use of the exact indicator $\eta^{\mbox{\scriptsize{H,Ex}}}_{\omega_i}$.
In the simulations, we take a snapshot space of dimension $7300$ 
giving errors of $0.05\%$ and $3.02\%$ in weighted $L^2$ and weighted $H^1$ norms, respectively. 
Thus, the solution $u_{\text{snap}}$ is as good as the fine-scale solution $u$. 
For the adaptive enrichment algorithm, the initial offline space has $4$
basis functions for each coarse grid node. 
At each enrichment (Step 4), we will add one basis function for the coarse grid nodes
selected in Step 3. 
We will terminate the iteration when the energy error $\| u - u_{\text{off}}\|_V$
is less than $5\%$ of $\| u - u_{\text{snap}}\|_V$.

%
\begin{table}[htb]
\centering
\begin{tabular}{|c|c|c|c|c|c|c|}
\hline 
\multirow{2}{*}{$\text{dim}(V_{\text{off}})$} &
\multicolumn{2}{c|}{  $\|u-u_{\text{off}} \|$ (\%) } &
\multicolumn{2}{c|}{  $\|u_{\text{snap}}-u_{\text{off}} \|$ (\%) }\\
\cline{2-5} {}&
$\hspace*{0.8cm}   L^{2}_\kappa(D)   \hspace*{0.8cm}$ &
$\hspace*{0.8cm}   H^{1}_\kappa(D)  \hspace*{0.8cm}$&
$\hspace*{0.8cm}   L^{2}_\kappa(D)   \hspace*{0.8cm}$ &
$\hspace*{0.8cm}   H^{1}_\kappa(D)  \hspace*{0.8cm}$\\
\hline\hline
       $1524$       &  $4.50$    & $31.29$ &  $4.49$    & $34.14$  \\
\hline
      $1711$      & $4.24$    & $27.37$ &  $4.23$    & $27.19$\\
\hline
      $2434$      & $2.34$    & $20.13$ &  $2.36$    & $20.31$\\
\hline

      $2637$      & $1.64$    & $15.43$ &  $1.61$    & $15.13$\\
\hline
       $3378$    & $0.54$  &$7.83$  &  $0.51$    & $7.22$\\
\hline
\end{tabular}
\caption{Convergence history for harmonic basis with $\theta=0.7$ and
$18$ iterations. The snapshot space has dimension $7300$ giving $0.05\%$ and $3.02\%$ weighted $L^2$ and weighted energy errors. When using 12 basis per coarse inner node, the weighted $L^2$ and the weighted $H^1$ errors will be $2.34\%$ and $19.77\%$, respectively, and the dimension of offline space is 4412.}
\label{table:Cross_theta.7_Harmonic}
\end{table}
%

%
\begin{table}[htb]
\centering
\begin{tabular}{|c|c|c|c|c|c|c|}
\hline 
\multirow{2}{*}{$\text{dim}(V_{\text{off}})$} &
\multicolumn{2}{c|}{  $\|u-u_{\text{off}} \|$ (\%) } &
\multicolumn{2}{c|}{  $\|u_{\text{snap}}-u_{\text{off}} \|$ (\%) }  \\
\cline{2-5} {}&
$\hspace*{0.8cm}   L^{2}_\kappa(D)   \hspace*{0.8cm}$ &
$\hspace*{0.8cm}   H^{1}_\kappa(D)  \hspace*{0.8cm}$&
$\hspace*{0.8cm}   L^{2}_\kappa(D)   \hspace*{0.8cm}$ &
$\hspace*{0.8cm}   H^{1}_\kappa(D)  \hspace*{0.8cm}$
\\
\hline\hline
       $1524$       &  $4.50$    & $31.29$ &  $4.49$    & $34.14$  \\
\hline
      $1646$      & $4.05$    & $26.80$ &  $4.04$    & $26.62$\\
\hline
      $1864$      & $3.09$    & $20.34$ &  $3.07$    & $20.11$\\
\hline

      $2220$      & $1.24$    & $14.43$ &  $1.20$    & $14.11$\\
\hline
       $3135$    & $0.48$  &$7.98$  &  $0.45$    & $7.39$\\
\hline
\end{tabular}
\caption{Convergence history for harmonic basis with $\theta=0.7$ and
$66$ iterations.
The number of iterations is $19$. The snapshot space has dimension $7300$ giving $0.05\%$ and $3.02\%$ weighted $L^2$ and weighted energy errors. 
When using 12 basis per coarse inner node, the weighted $L^2$ and the weighted $H^1$ errors will be $2.34\%$ and $19.77\%$, respectively, and the dimension of offline space is 4412.}
\label{table:Cross_theta.2_Harmonic}
\end{table}

In Table \ref{table:Cross_theta.7_Harmonic} and Table \ref{table:Cross_theta.2_Harmonic},
we present the convergence history of the adaptive enrichment algorithm
for $\theta=0.7$ and $\theta=0.2$ respectively. 
For both cases, we see a convergence of the algorithm.
For the case $\theta=0.7$, the algorithm requires $18$ iterations to achieve the desired accuracy. 
The dimension of the corresponding offline space is $3378$.
Moreover, the error $u-u_{\text{off}}$ in relative weighted $L^2$ and energy norms
are $0.54\%$ and $7.83\%$ respectively, while the 
error $u_{\text{snap}}-u_{\text{off}}$ in relative weighted $L^2$ and energy norms
are $0.51\%$ and $7.22\%$ respectively.
And we see the similarity of the errors $u-u_{\text{off}}$ and $u_{\text{snap}}-u_{\text{off}}$.
For the case $\theta=0.2$, the algorithm requires $66$ iterations to achieve the desired accuracy. 
The dimension of the corresponding offline space is $3135$.
Moreover, the error $u-u_{\text{off}}$ in relative weighted $L^2$ and energy norms
are $0.48\%$ and $7.98\%$ respectively, while the 
error $u_{\text{snap}}-u_{\text{off}}$ in relative weighted $L^2$ and energy norms
are $0.45\%$ and $7.39\%$ respectively.
Furthermore,  we observe that the use of $\theta=0.2$ 
gives the same level of error for a smaller offline space compared with $\theta=0.7$. 
Thus, we conclude that a smaller value of $\theta$ will give a more economical offline space. 
To show that our adaptive enrichment algorithm gives a more efficient scheme, 
we report some computational results with uniform enrichment. 
In this case, we use $12$ basis functions for each interior coarse grid node giving 
an offline space of dimension $4412$. 
The relative weighted $L^2$ and energy errors are $2.32\%$ and $19.53\%$ respectively. 
From this result, we see that our adaptive enrichment algorithm
gives a smaller offline space and at the same time a better accuracy
than a scheme with uniform number of basis functions. 


%
\begin{table}[htb]
\centering
\begin{tabular}{|c|c|c|c|c|c|c|}
\hline 
\multirow{2}{*}{$\text{dim}(V_{\text{off}})$} &
\multicolumn{2}{c|}{  $\|u-u^{\text{off}} \|$ (\%) } &
\multicolumn{2}{c|}{  $\|u^{\text{snap}}-u^{\text{off}} \|$ (\%) }  \\
\cline{2-5} {}&
$\hspace*{0.8cm}   L^{2}_\kappa(D)   \hspace*{0.8cm}$ &
$\hspace*{0.8cm}   H^{1}_\kappa(D)  \hspace*{0.8cm}$&
$\hspace*{0.8cm}   L^{2}_\kappa(D)   \hspace*{0.8cm}$ &
$\hspace*{0.8cm}   H^{1}_\kappa(D)  \hspace*{0.8cm}$
\\
\hline\hline
       $1524$       &  $4.50$    & $31.29$ &  $4.49$    & $34.14$  \\
\hline
      $1762$      & $3.96$    & $27.09$ &  $3.95$    & $26.91$\\
\hline
      $2333$      & $2.07$    & $19.00$ &  $2.04$    & $18.75$\\
\hline

      $2522$      & $1.38$    & $15.12$ &  $1.36$    & $14.81$\\
\hline
       $3466$    & $0.46$  &$7.52$  &  $0.44$    & $6.89$\\
\hline
\end{tabular}
\caption{Convergence history for harmonic basis with $\theta=0.7$ and the exact indicator.
The number of iterations is $23$. The snapshot space has dimension $7300$ giving $0.05\%$ and $3.02\%$ weighted $L^2$ and weighted energy errors. 
}
\label{table:Cross_theta.7_Harmonic_Uniform}
\end{table}
To test the reliability and efficiency of the proposed indicator, we apply the adaptive enrichment algorithm
with the exact energy error as indicator and $\theta=0.7$. The results are shown 
in Table \ref{table:Cross_theta.7_Harmonic_Uniform}.
In particular, the algorithm requires $19$ iterations to achieve the desired accuracy. 
The dimension of the corresponding offline space is $3466$.
Moreover, the error $u-u_{\text{off}}$ in relative weighted $L^2$ and energy norms
are $0.46\%$ and $7.52\%$ respectively,
while the 
error $u_{\text{snap}}-u_{\text{off}}$ in relative weighted $L^2$ and energy norms
are $0.44\%$ and $6.89\%$ respectively.
Comparing the results in Table \ref{table:Cross_theta.7_Harmonic} and Table \ref{table:Cross_theta.7_Harmonic_Uniform}
for the use of the proposed and the exact indicator respectively, 
we see that both indicators give similar convergence behavior and offline space dimensions. 


%
\begin{figure}[htb]
 \centering
 \subfigure[Proposed indicator with $\theta=0.7$]{\label{fig:HarmonicBasis_NoOfBasisPerCoarseNode_Cross}
    \includegraphics[width = 0.30\textwidth, keepaspectratio = true]{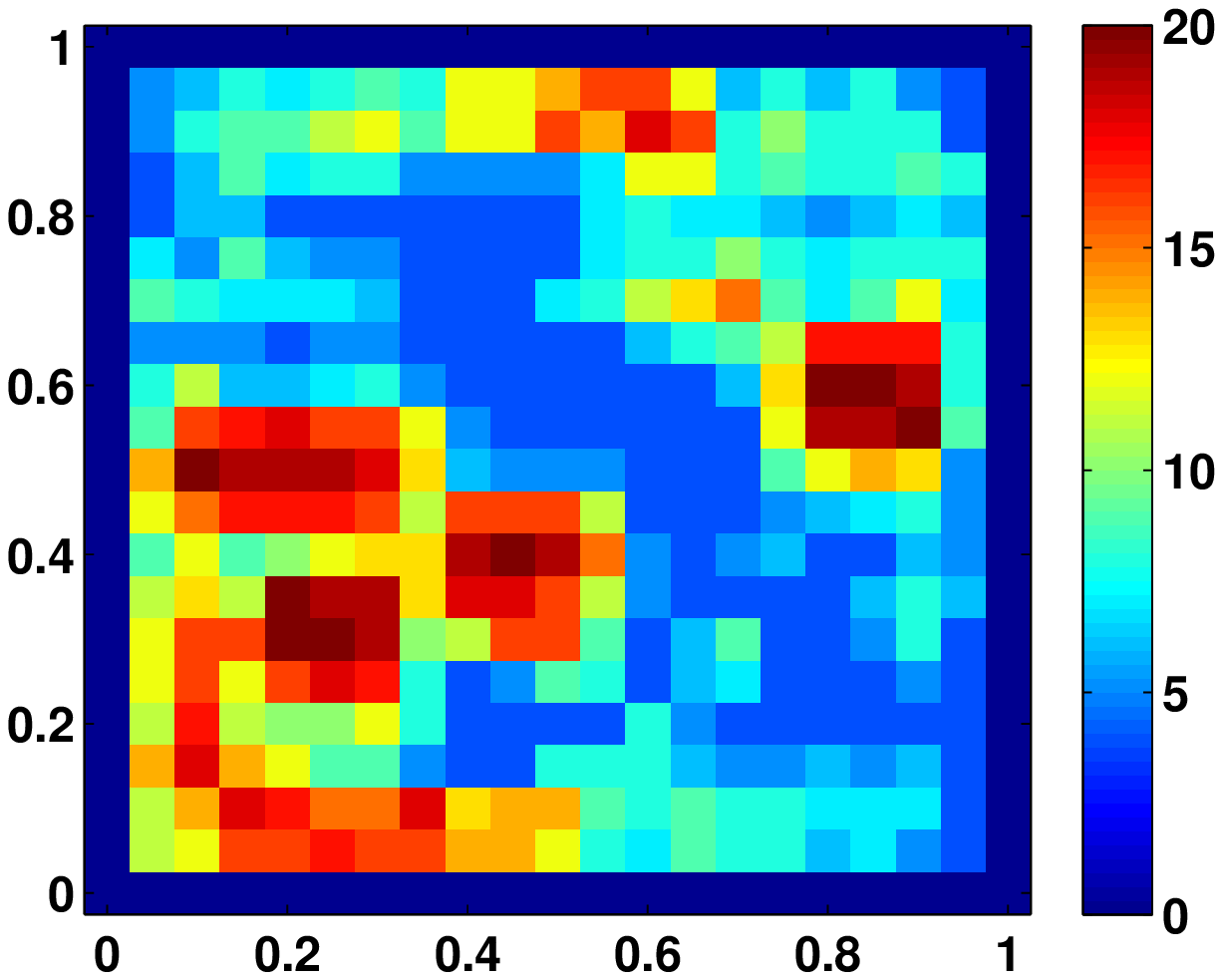}
   }
  \subfigure[Proposed indicator with $\theta=0.2$]{\label{fig:HarmonicBasis_NoOfBasisPerCoarseNode_Cross_theta.2}
     \includegraphics[width = 0.30\textwidth, keepaspectratio = true]{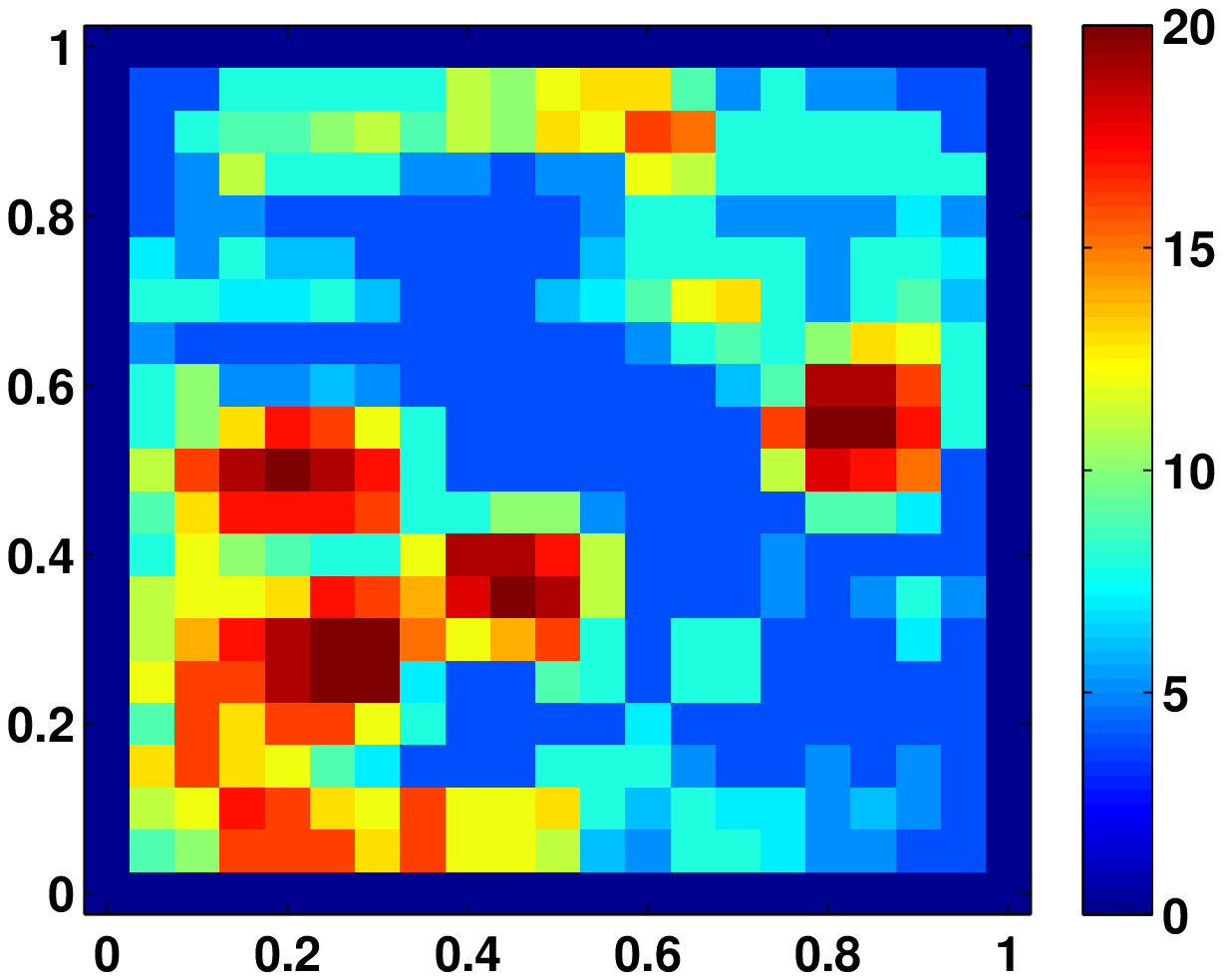}
  }
 \subfigure[Exact indicator with $\theta=0.7$]{\label{fig:HarmonicBasis_NoOfBasisPerCoarseNode_Cross_Uniform}
     \includegraphics[width = 0.30\textwidth, keepaspectratio = true]{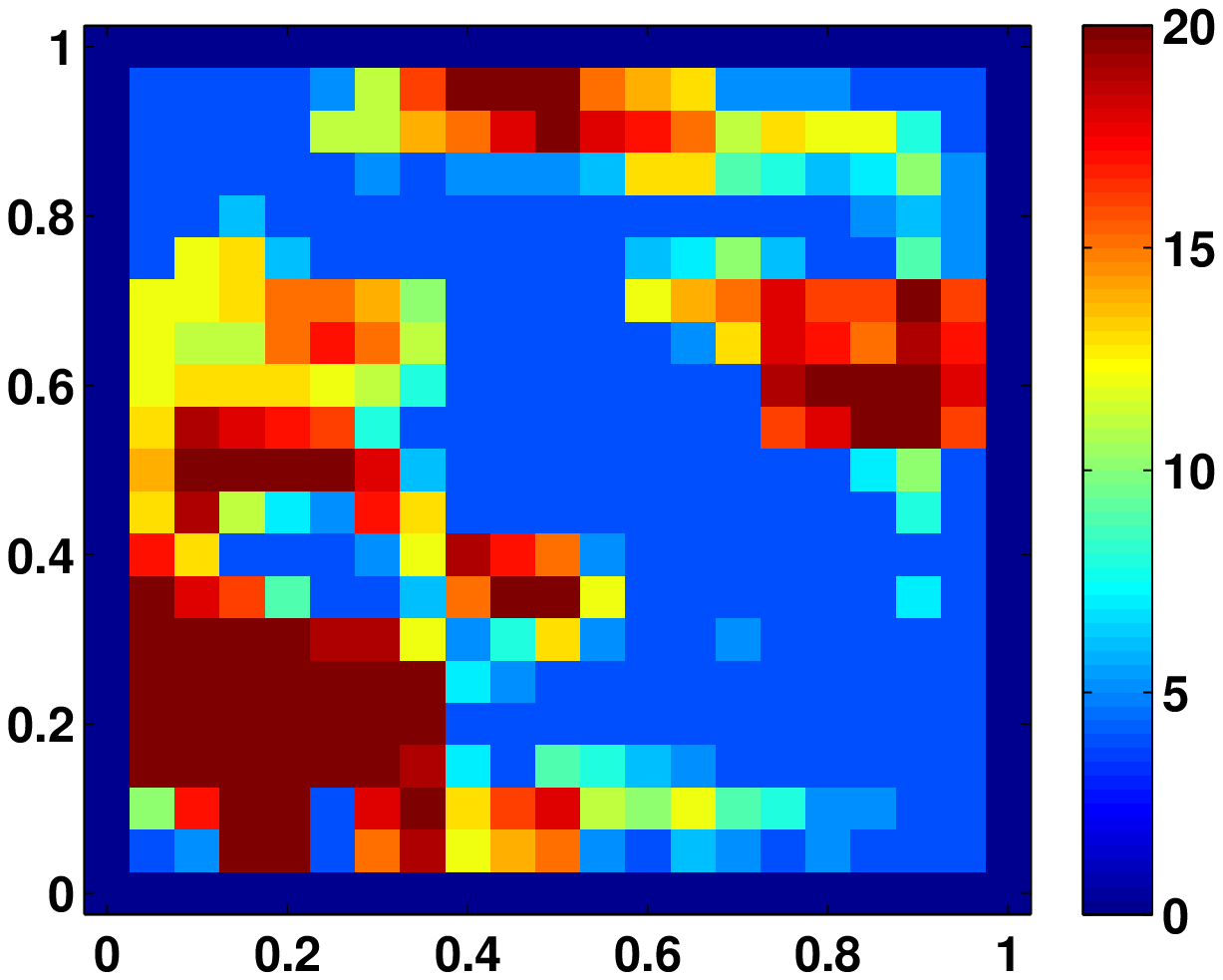}
  }
 \caption{Dimension distributions of the last offline space for harmonic basis with permeability field \ref{fig:newperm2_cross2}.}
\label{fig:BasisPerNode_Harmonic_Cross}
\end{figure}
In Figure \ref{fig:BasisPerNode_Harmonic_Cross}, we display the number of basis functions for each coarse grid node
of the last offline spaces 
for the proposed indicator with $\theta=0.7$, the proposed indicator with $\theta=0.2$ and the exact indicator with $\theta=0.7$.
From Figures \ref{fig:HarmonicBasis_NoOfBasisPerCoarseNode_Cross} and \ref{fig:HarmonicBasis_NoOfBasisPerCoarseNode_Cross_theta.2},
we observe a similar dimension distribution for the use of the proposed indicator with $\theta=0.7$ and $\theta=0.2$,
and the case $\theta=0.2$ gives a smaller number of basis functions. 
For the case with the exact indicator, we see from Figure \ref{fig:HarmonicBasis_NoOfBasisPerCoarseNode_Cross_Uniform}
that the dimension distribution follows a similar pattern, but with regions that contain larger number of basis functions. 


%
\begin{figure}[htb]
 \centering
\subfigure[Proposed indicator with the last offline space]{\label{fig:Hm_ErrorDistri_Cross_larger}\includegraphics[width = 0.22\textwidth, keepaspectratio = true]{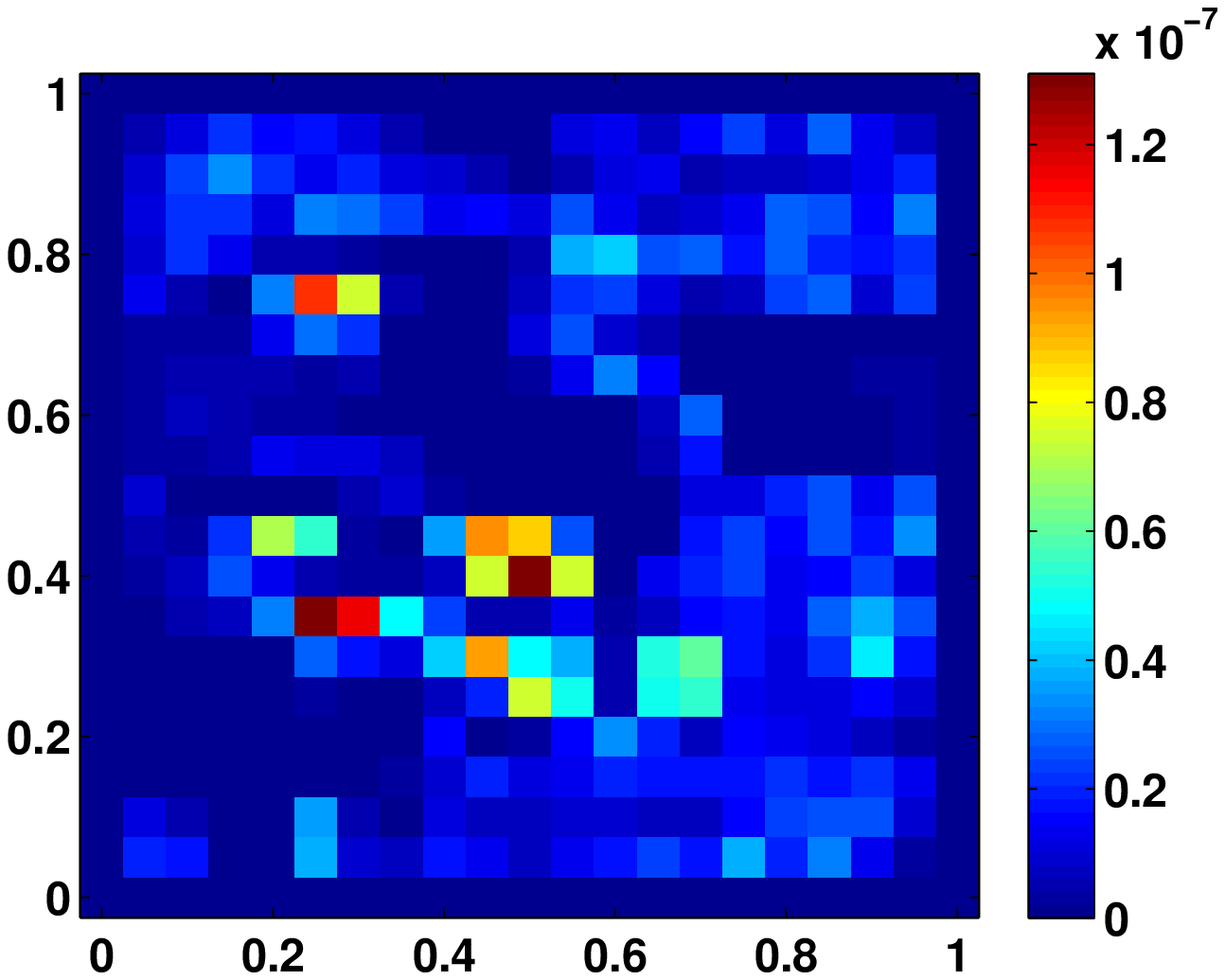}}
\subfigure[Proposed indicator with an intermediate offline space]{\label{fig:Hm_ErrorDistri_Cross_smaller}\includegraphics[width = 0.22\textwidth, keepaspectratio = true]{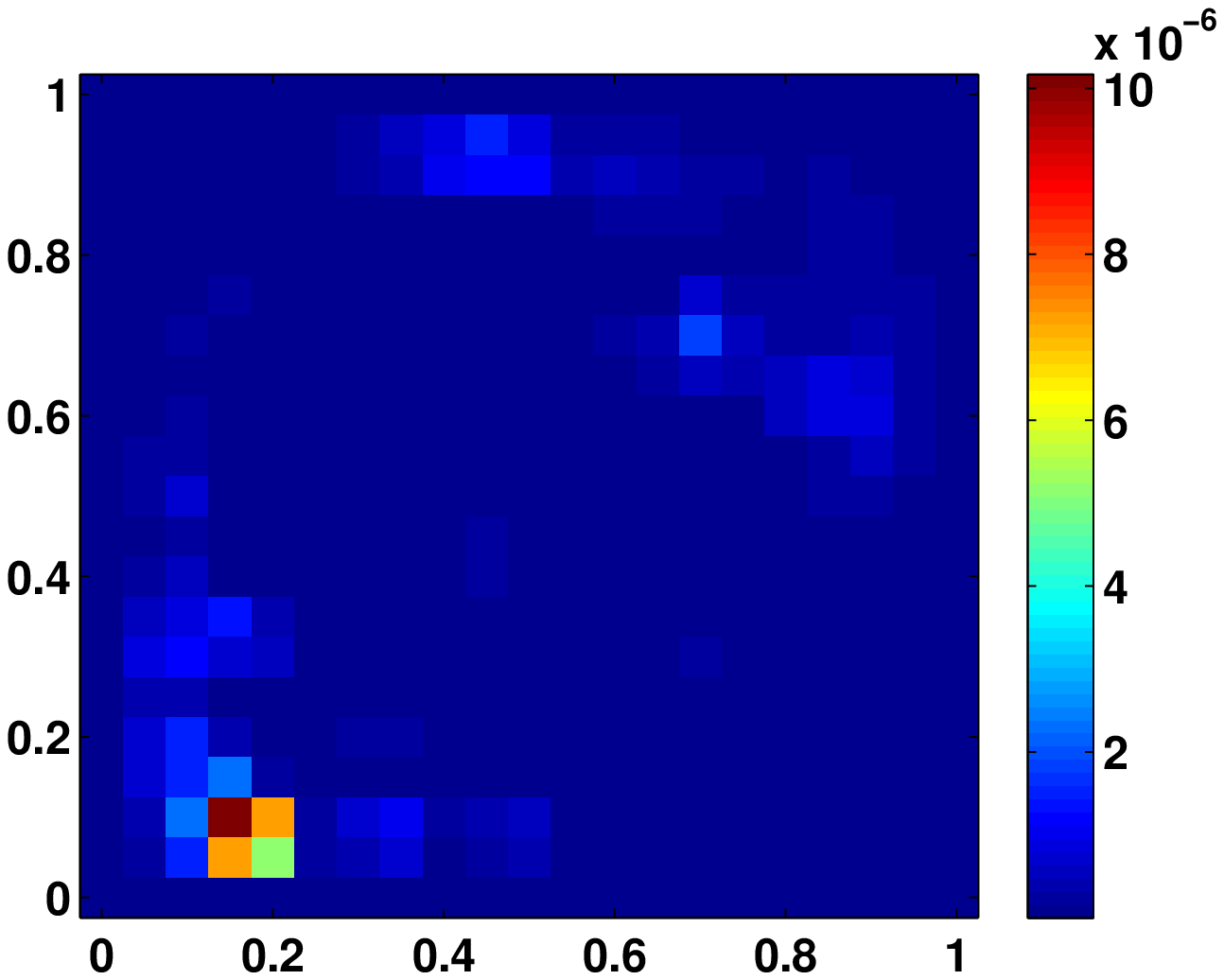}}
    \subfigure[Exact indicator with the last offline space]{\label{fig:Hm_ErrorDistri_Cross_U_larger}\includegraphics[width = 0.22\textwidth, keepaspectratio = true]{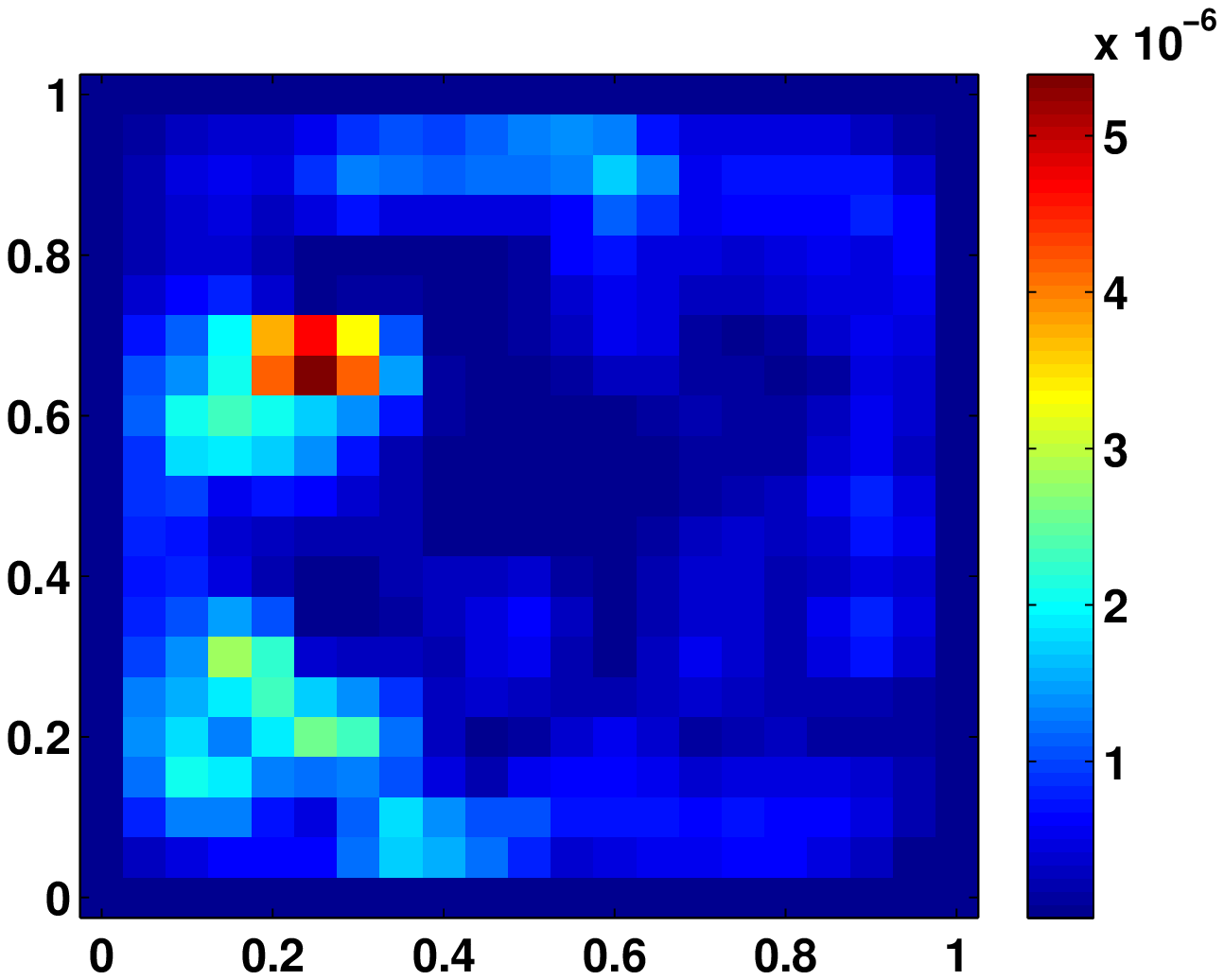}}
\subfigure[Exact indicator with an intermediate offline space]{\label{fig:Hm_ErrorDistri_Cross_U_smaller}\includegraphics[width = 0.22\textwidth, keepaspectratio = true]{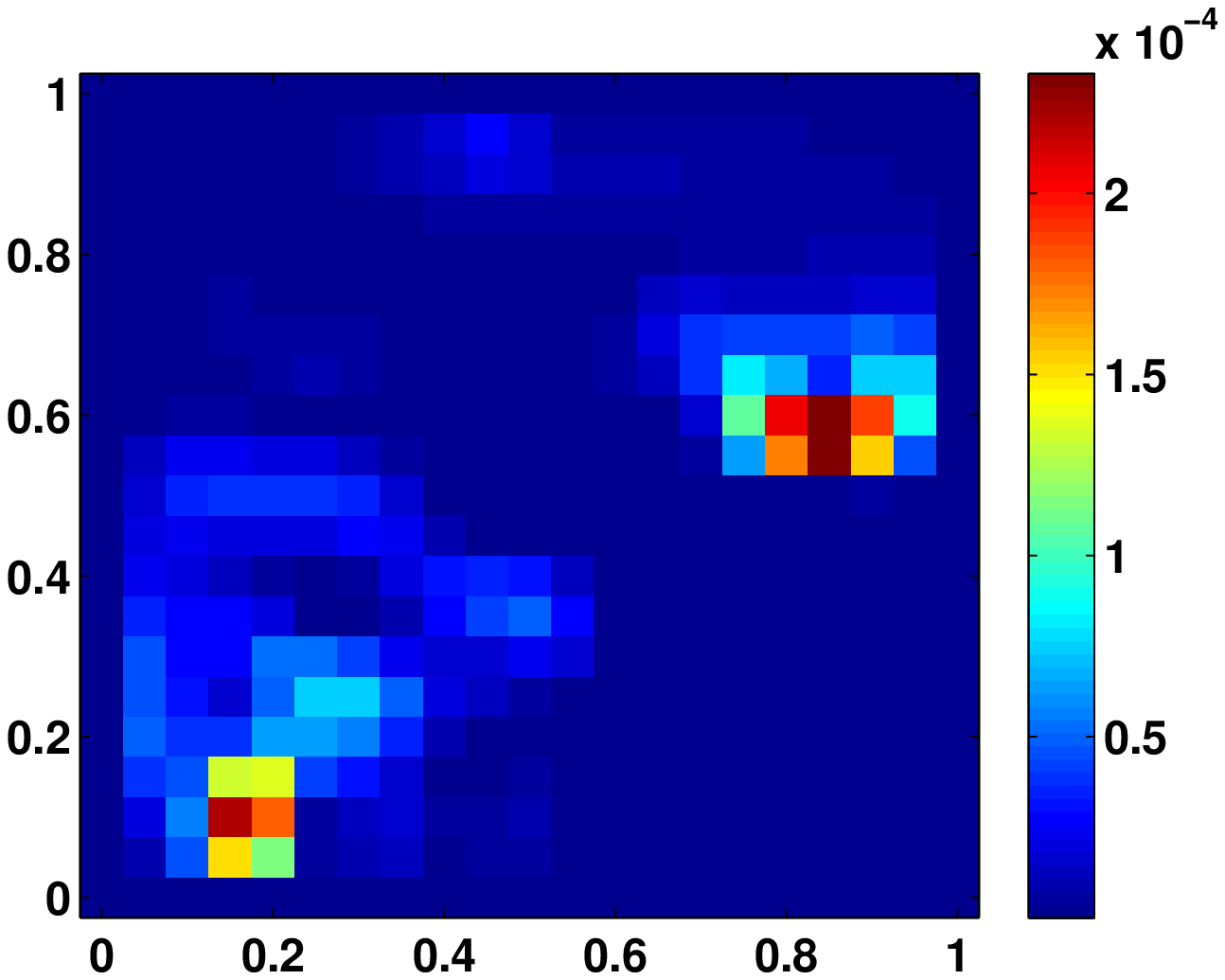}}
\caption{Coarse-grid energy error distribution using harmonic basis with permeability field \ref{fig:newperm2_cross2}.}
\label{fig:ErrorDistri_Hm_Cross}
\end{figure}

Finally, 
we present the energy errors on coarse neighborhoods 
for $\theta=0.7$ for an intermediate offline space and the last offline space of the 
proposed indicator $\eta^{\mbox{\scriptsize{H,En}}}_{\omega_i}$ and the exact indicator $\eta^{\mbox{\scriptsize{H,Ex}}}_{\omega_i}$.
In Figures \ref{fig:Hm_ErrorDistri_Cross_larger} and \ref{fig:Hm_ErrorDistri_Cross_smaller},
the energy error distributions for the last offline spaces and an intermediate offline space
obtained by the proposed indicator are shown respectively. 
We see how the energy error is reduced by enriching the space 
from an intermediate step to the final step. 
A similar situation is also seen from Figures \ref{fig:Hm_ErrorDistri_Cross_U_larger} and \ref{fig:Hm_ErrorDistri_Cross_U_smaller}
for the case with the exact indicator.


%
\subsection{Numerical results with spectral basis}
\label{sec:622}
In this section,
we repeat the above tests using the spectral snapshot space instead of the harmonic snapshot space with the proposed indicator $\eta^{\mbox{\scriptsize{U,En}}}_{\omega_i}$ and the exact indicator $\eta^{\mbox{\scriptsize{U,Ex}}}_{\omega_i}$. The results are presented in Tables \ref{table:Cross_theta.7_Spectral}, \ref{table:Cross_theta.2_Spectral} and \ref{table:Cross_theta.7_Spectral_Uniform}. 
In the simulations, we take a snapshot space of dimension $3690$ 
giving errors of $0.01\%$ and $2.84\%$ in weighted $L^2$ and  energy norms respectively. 
Thus, the solution $u_{\text{snap}}$ is as good as the fine-scale solution $u$. 
For the adaptive enrichment algorithm, the initial offline space has $2$
basis functions for each coarse grid node. 
At each enrichment (Step 4), we will add one basis function for the coarse grid nodes
selected in Step 3. 
We will terminate the iteration when the energy error $\| u - u_{\text{off}}\|_V$
is less than $5\%$ of $\| u - u_{\text{snap}}\|_V$.


%
\begin{table}[htb]
\centering
\begin{tabular}{|c|c|c|c|c|c|c|}
\hline 
\multirow{2}{*}{$\text{dim}(V_{\text{off}})$} &
\multicolumn{2}{c|}{  $\|u-u^{\text{off}} \|$ (\%) } &
\multicolumn{2}{c|}{  $\|u^{\text{snap}}-u^{\text{off}} \|$ (\%) }  \\
\cline{2-5} {}&
$\hspace*{0.8cm}   L^{2}_\kappa(D)   \hspace*{0.8cm}$ &
$\hspace*{0.8cm}   H^{1}_\kappa(D)  \hspace*{0.8cm}$&
$\hspace*{0.8cm}   L^{2}_\kappa(D)   \hspace*{0.8cm}$ &
$\hspace*{0.8cm}   H^{1}_\kappa(D)  \hspace*{0.8cm}$
\\
\hline\hline
       $802$       &  $0.87$    & $20.15$ &  $0.87$    & $19.94$  \\
\hline
      $868$      & $0.83$    & $16.51$ &  $0.83$    & $16.26$\\
\hline
       $979$       &  $0.33$    & $12.62$ &  $0.33$    & $12.30$  \\
\hline
      $1106$      & $0.32$    & $10.44$ &  $0.32$    & $10.05$\\
\hline
       $1410$       &  $0.10$    & $7.43$ &  $0.10$    & $6.87$  \\
\hline
\end{tabular}
\caption{Convergence history for spectral basis with $\theta=0.7$ and $5$ iterations.
The snapshot space has dimension $3690$ giving $0.01\%$ and $2.84\%$ weighted $L^2$ and weighted energy errors. When using $5$ basis per interior coarse node, the weighted $L^2$ and the weighted energy errors will be $0.09\%$ and $7.40\%$, respectively, and the dimension of offline space is $1885$.}
\label{table:Cross_theta.7_Spectral}
\end{table}
\begin{table}[htb]
\centering
\begin{tabular}{|c|c|c|c|c|c|c|}
\hline 
\multirow{2}{*}{$\text{dim}(V_{\text{off}})$} &
\multicolumn{2}{c|}{  $\|u-u^{\text{off}} \|$ (\%) } &
\multicolumn{2}{c|}{  $\|u^{\text{snap}}-u^{\text{off}} \|$ (\%) }  \\
\cline{2-5} {}&
$\hspace*{0.8cm}   L^{2}_\kappa(D)   \hspace*{0.8cm}$ &
$\hspace*{0.8cm}   H^{1}_\kappa(D)  \hspace*{0.8cm}$&
$\hspace*{0.8cm}   L^{2}_\kappa(D)   \hspace*{0.8cm}$ &
$\hspace*{0.8cm}   H^{1}_\kappa(D)  \hspace*{0.8cm}$
\\
\hline\hline
       $802$       &  $0.87$    & $20.15$ &  $0.87$    & $19.94$  \\
\hline
      $856$      & $0.83$    & $16.25$ &  $0.82$    & $16.00$\\
\hline
       $960$       &  $0.34$    & $12.58$ &  $0.33$    & $12.26$  \\
\hline
      $1116$      & $0.32$    & $10.27$ &  $0.32$    & $9.87$\\
\hline
       $1334$       &  $0.09$    & $7.55$ &  $0.09$    & $6.99$  \\
\hline
\end{tabular}
\caption{Convergence history for spectral basis with $\theta=0.7$ and $19$ iterations.
The snapshot space has dimension $3690$ giving $0.01\%$ and $2.84\%$ weighted $L^2$ and weighted energy errors. When using $5$ basis per interior coarse node, the weighted $L^2$ and the weighted energy errors will be $0.09\%$ and $7.40\%$, respectively, and the dimension of offline space is $1885$.}
\label{table:Cross_theta.2_Spectral}
\end{table}

In Table \ref{table:Cross_theta.7_Spectral} and Table \ref{table:Cross_theta.2_Spectral},
we present the convergence history of the adaptive enrichment algorithm
for $\theta=0.7$ and $\theta=0.2$ respectively. 
For both cases, we see a clear convergence of the algorithm.
For the case $\theta=0.7$, the algorithm requires $5$ iterations to achieve the desired accuracy. 
The dimension of the corresponding offline space is $1410$.
Moreover, the error $u-u_{\text{off}}$ in relative weighted $L^2$ and energy norms
are $0.10\%$ and $7.43\%$ respectively, while the 
error $u_{\text{snap}}-u_{\text{off}}$ in relative weighted $L^2$ and energy norms
are $0.10\%$ and $6.87\%$ respectively.
For the case $\theta=0.2$, the algorithm requires $19$ iterations to achieve the desired accuracy. 
The dimension of the corresponding offline space is $1334$.
Moreover, the error $u-u_{\text{off}}$ in relative weighted $L^2$ and energy norms
are $0.09\%$ and $7.55\%$ respectively, while the 
error $u_{\text{snap}}-u_{\text{off}}$ in relative weighted $L^2$ and energy norms
are $0.09\%$ and $6.99\%$ respectively.
Furthermore,  we observe that the use of $\theta=0.2$ 
gives the same level of error for a smaller offline space compared with $\theta=0.2$. 
Thus, we conclude that a smaller value of $\theta$ will give a more economical offline space. 
To show that our adaptive enrichment algorithm gives a more efficient scheme, 
we report some computational results with uniform enrichment. 
In this case, we use $5$ basis functions for each interior coarse grid node giving 
an offline space of dimension $1885$. 
The relative weighted $L^2$ and energy errors are $0.09\%$ and $7.40\%$ respectively. 
From this result, we see that our adaptive enrichment algorithm
gives a smaller offline space and at the same time a better accuracy
than a scheme with uniform number of basis functions.

\begin{table}[htb]
\centering
\begin{tabular}{|c|c|c|c|c|c|c|}
\hline 
\multirow{2}{*}{$\text{dim}(V_{\text{off}})$} &
\multicolumn{2}{c|}{  $\|u-u^{\text{off}} \|$ (\%) } &
\multicolumn{2}{c|}{  $\|u^{\text{snap}}-u^{\text{off}} \|$ (\%) }  \\
\cline{2-5} {}&
$\hspace*{0.8cm}   L^{2}_\kappa(D)   \hspace*{0.8cm}$ &
$\hspace*{0.8cm}   H^{1}_\kappa(D)  \hspace*{0.8cm}$&
$\hspace*{0.8cm}   L^{2}_\kappa(D)   \hspace*{0.8cm}$ &
$\hspace*{0.8cm}   H^{1}_\kappa(D)  \hspace*{0.8cm}$
\\
\hline\hline
       $802$       &  $0.87$    & $20.15$ &  $0.87$    & $19.94$  \\
\hline
       $884$       & $0.42$    & $14.73$ &  $0.42$    & $14.45$  \\
\hline
      $1000$      & $0.18$    & $12.25$ &  $0.18$    & $11.91$\\
\hline
      $1119$      & $0.17$    & $9.83$ &  $0.17$    & $9.41$\\
\hline
      $1392$      & $0.10$    & $7.12$ &  $0.10$    & $6.53$  \\
\hline
\end{tabular}
\caption{Convergence history for spectral basis with $\theta=0.7$ and the exact indicator.
The number of iteration is $6$.
The snapshot space has dimension $3690$ giving $0.01\%$ and $2.84\%$ weighted $L^2$ and weighted energy errors. When using $5$ basis per interior coarse node, the weighted $L^2$ and the weighted energy errors will be $0.09\%$ and $7.40\%$, respectively, and the dimension of offline space is $1885$.
}
\label{table:Cross_theta.7_Spectral_Uniform}
\end{table}

To test the reliability and efficiency of the proposed indicator, we apply the adaptive enrichment algorithm
with the exact energy error as indicator and $\theta=0.7$. The results are shown 
in Table \ref{table:Cross_theta.7_Spectral_Uniform}.
In particular, the algorithm requires $6$ iterations to achieve the desired accuracy. 
The dimension of the corresponding offline space is $1392$.
Moreover, the error $u-u_{\text{off}}$ in relative weighted $L^2$ and energy norms
are $0.10\%$ and $7.12\%$ respectively,
while the 
error $u_{\text{snap}}-u_{\text{off}}$ in relative weighted $L^2$ and energy norms
are $0.10\%$ and $6.53\%$ respectively.
Comparing the results in Table \ref{table:Cross_theta.7_Spectral} and Table \ref{table:Cross_theta.7_Spectral_Uniform}
for the use of the proposed and the exact indicator respectively, 
we see that both indicators give similar convergence behavior and offline space dimensions. 
We also observe that the exact indicator performs better
in this case.

\begin{figure}[htb]
 \centering
\subfigure[Proposed indicator with $\theta=0.7$]{\label{fig:SpectralBasis_NoOfBasisPerCoarseNode_Cross}\includegraphics[width = 0.30\textwidth, keepaspectratio = true]{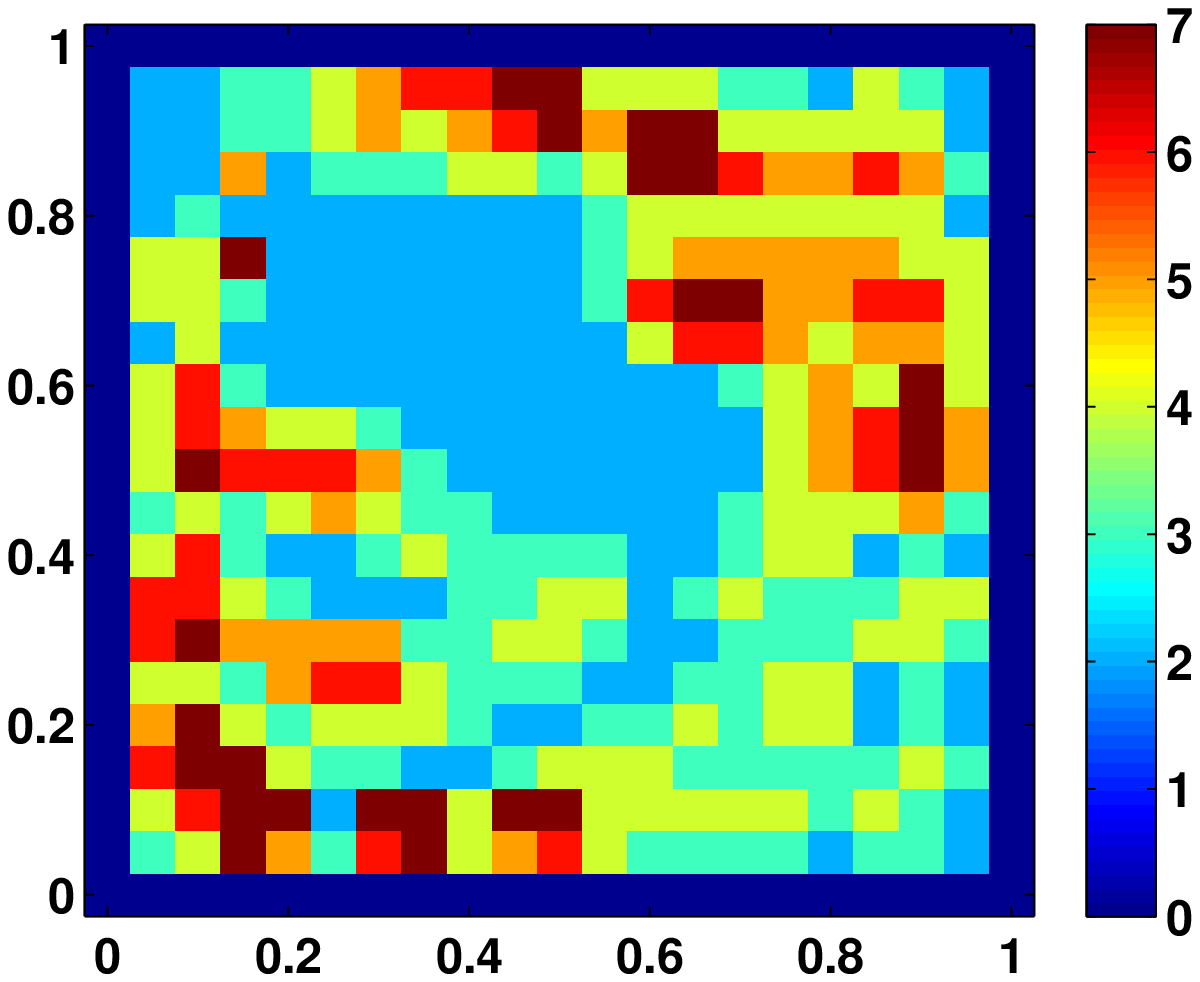}}
\subfigure[Proposed indicator with $\theta=0.2$]{\label{fig:SpectralBasis_NoOfBasisPerCoarseNode_Cross_theta.2}\includegraphics[width = 0.30\textwidth, keepaspectratio = true]{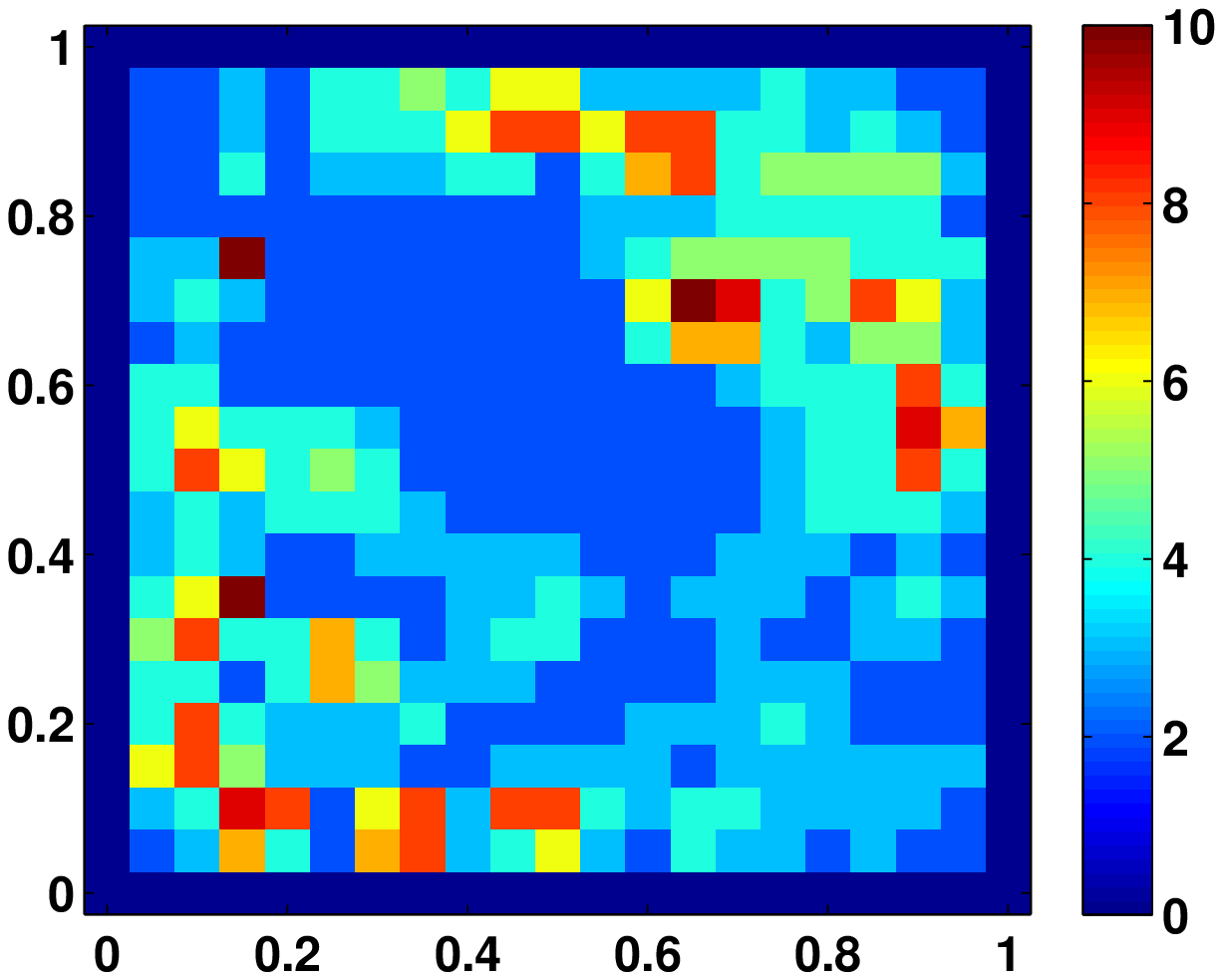}}
\subfigure[Exact indicator with $\theta=0.7$]{\label{fig:SpectralBasis_NoOfBasisPerCoarseNode_Cross_Uniform}\includegraphics[width = 0.30\textwidth, keepaspectratio = true]{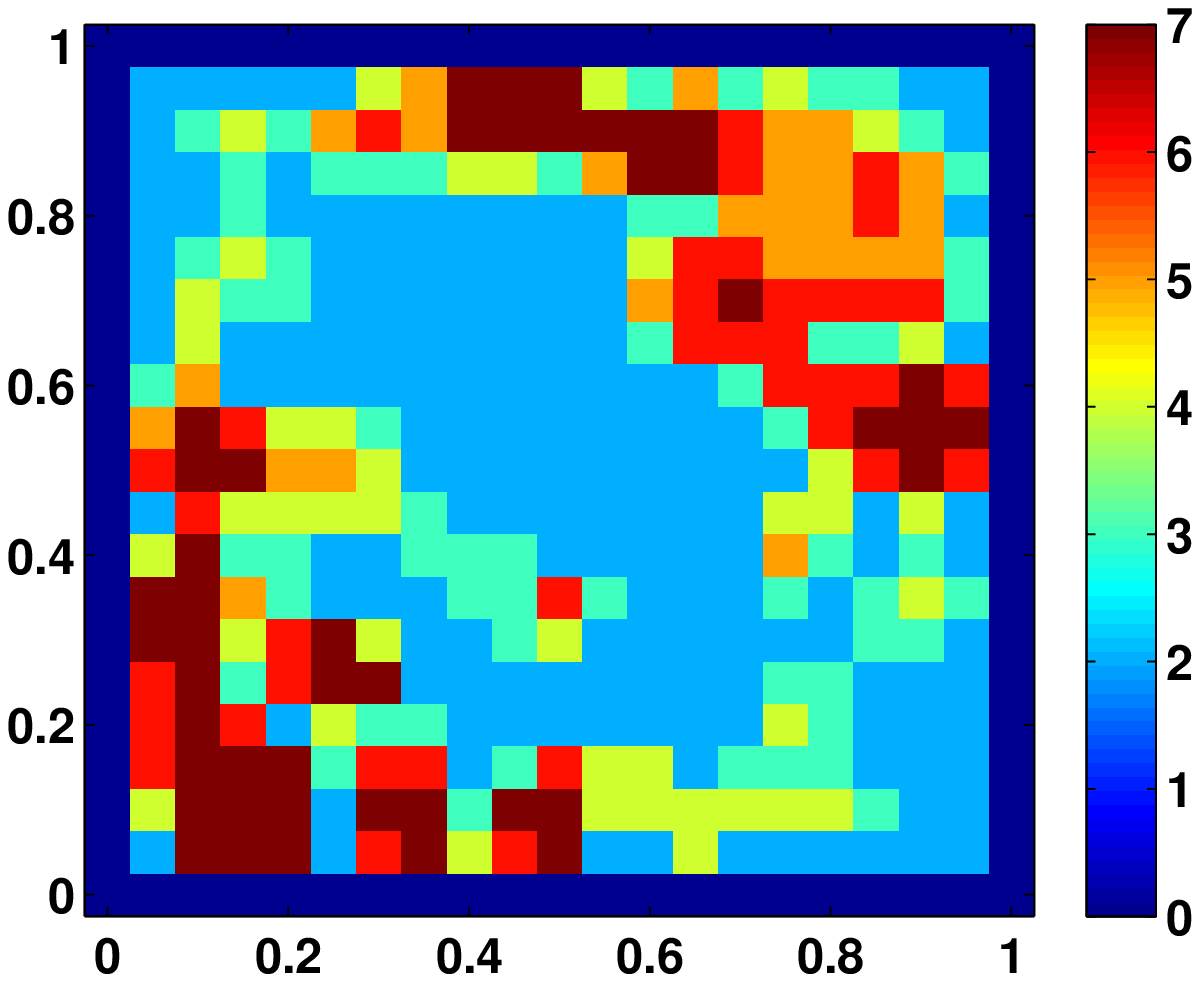}}
 \caption{Dimension distributions of the last offline space for spectral basis 
with permeability field \ref{fig:newperm2_cross2}.}
\label{fig:BasisPerNode_spectral_Cross}
\end{figure}
In Figure \ref{fig:BasisPerNode_spectral_Cross}, we display the number of basis functions for each coarse grid node
of the last offline spaces 
for the proposed indicator with $\theta=0.7$, the proposed indicator with $\theta=0.2$ and the exact indicator with $\theta=0.7$.
From Figures \ref{fig:SpectralBasis_NoOfBasisPerCoarseNode_Cross} and \ref{fig:SpectralBasis_NoOfBasisPerCoarseNode_Cross_theta.2},
we observe a similar dimension distribution for the use of the proposed indicator with $\theta=0.7$ and $\theta=0.2$,
and the case $\theta=0.2$ gives a smaller number of basis functions. 
For the case with the exact indicator, we see from Figure \ref{fig:SpectralBasis_NoOfBasisPerCoarseNode_Cross_Uniform}
that the dimension distribution follows a similar pattern, but with regions that contain larger number of basis functions.

\begin{figure}[htb!]
 \centering
 \subfigure[Proposed indicator with the last offline space]{\label{fig:Spectral_ErrorDistri_Cross_larger}
\includegraphics[width = 0.22\textwidth, keepaspectratio = true]{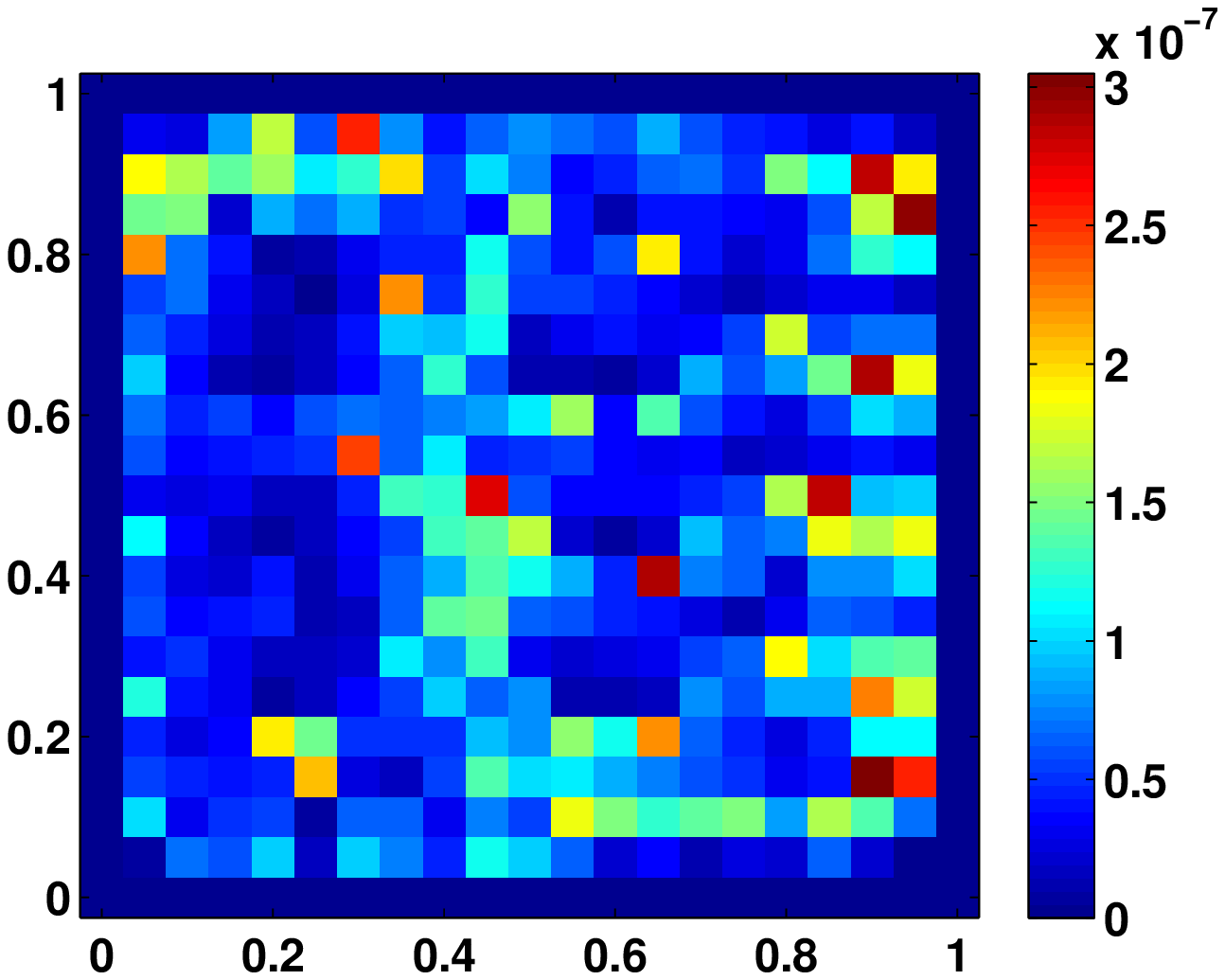}}
\subfigure[Proposed indicator with an intermediate offline space]{\label{fig:Spectral_ErrorDistri_Cross_smaller}\includegraphics[width = 0.22\textwidth, keepaspectratio = true]{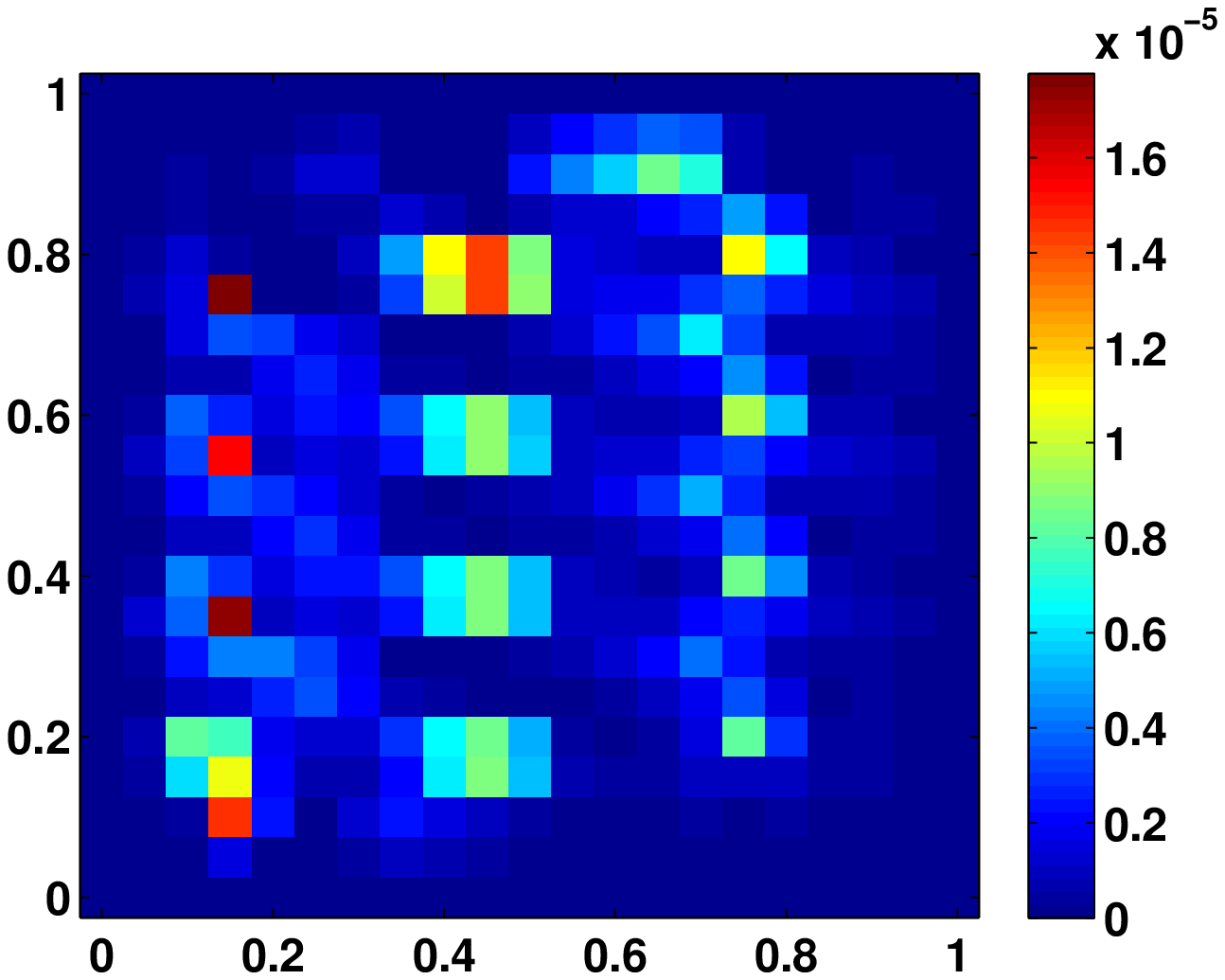}}
    \subfigure[Exact indicator with the last offline space]{\label{fig:Spectral_ErrorDistri_Cross_U_larger}
     \includegraphics[width = 0.22\textwidth, keepaspectratio = true]{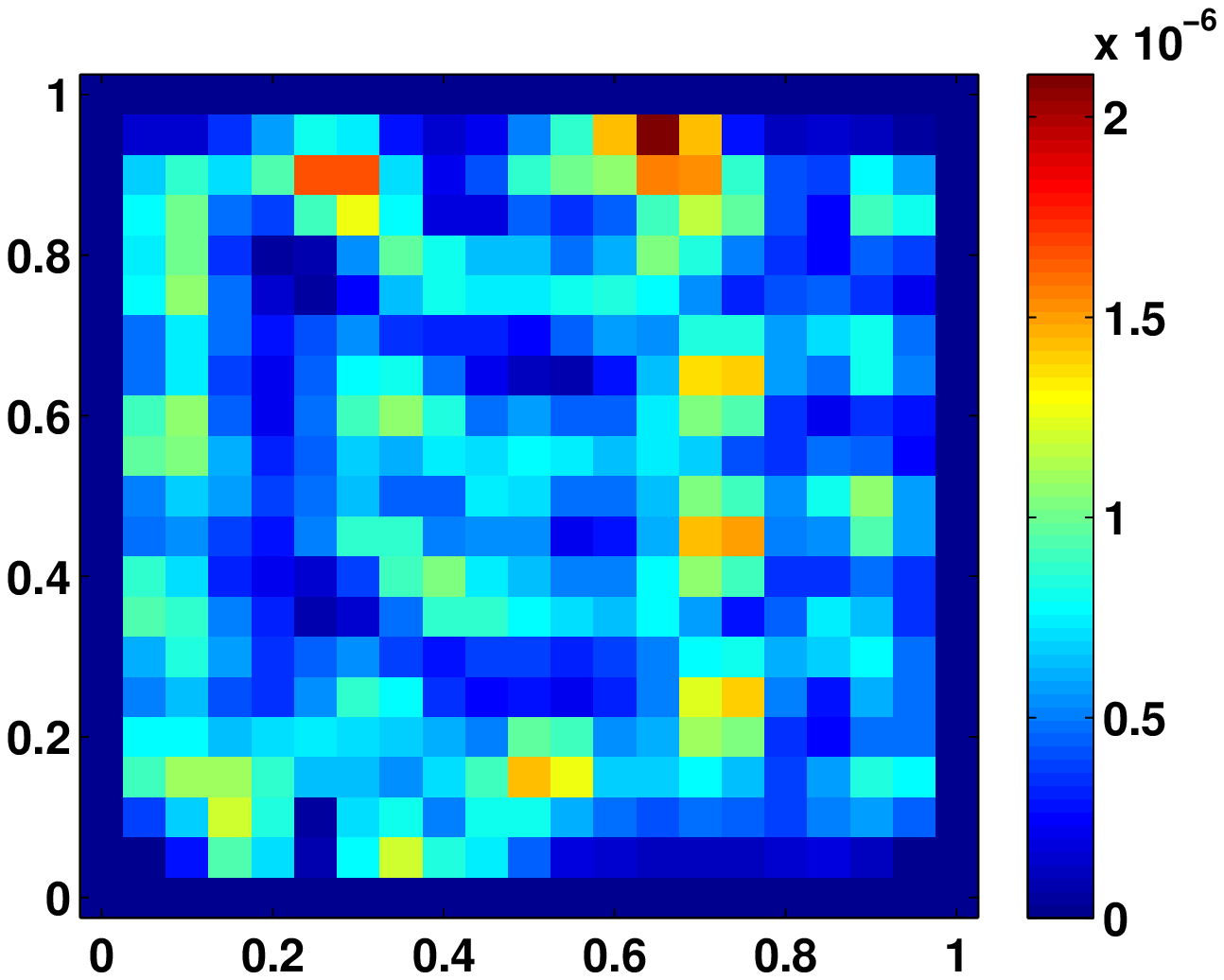}}
\subfigure[Exact indicator with an intermediate offline space]{\label{fig:Spectral_ErrorDistri_Cross_U_smaller}
     \includegraphics[width = 0.22\textwidth, keepaspectratio = true]{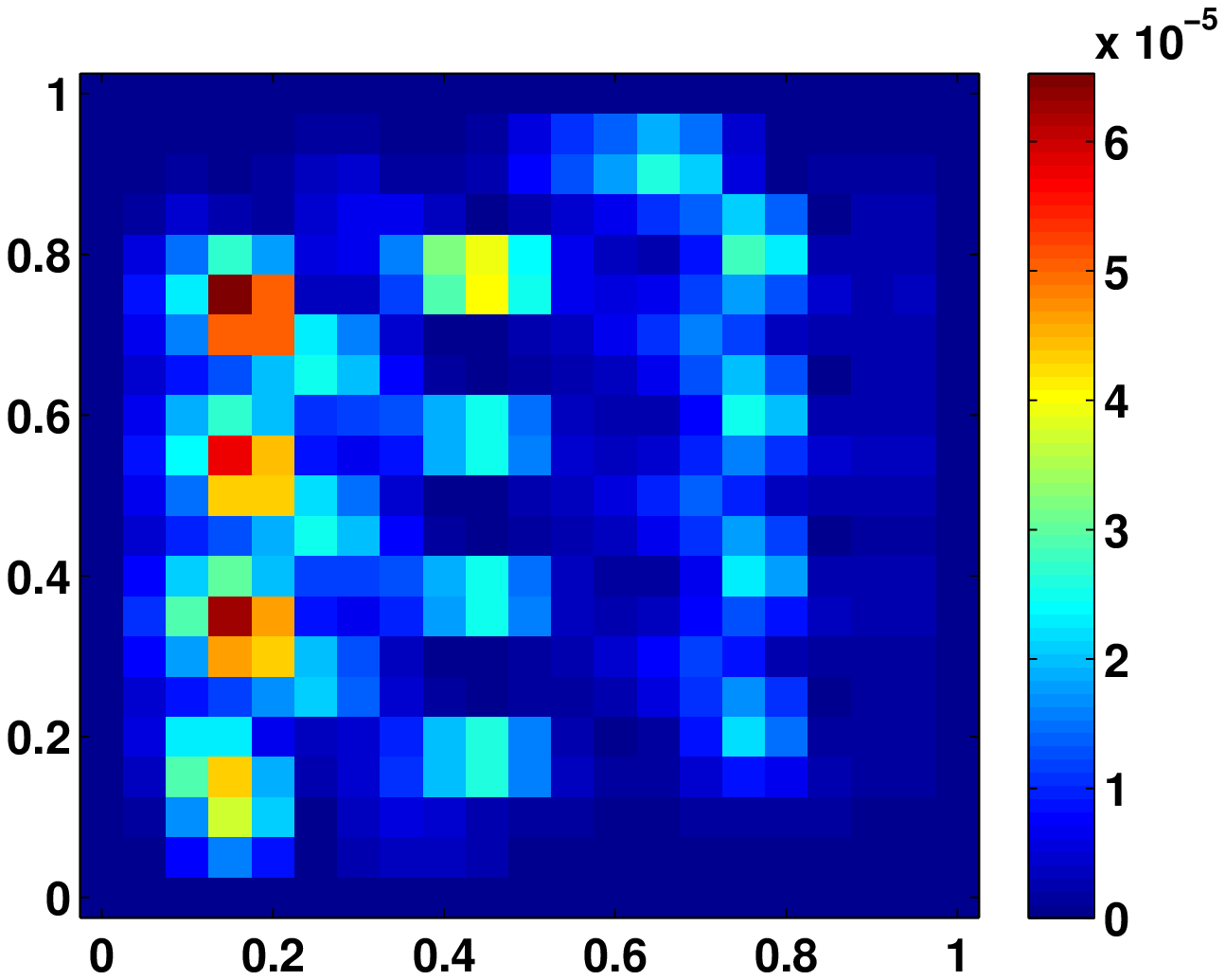}}
 \caption{Coarse-grid energy error distribution using spectral basis with permeability field \ref{fig:newperm2_cross2}.}
\label{fig:ErrorDistri_Spectral_Cross}
\end{figure}

Finally, 
we present the energy errors on coarse neighborhoods 
for $\theta=0.7$ for an intermediate offline space and the last offline space of the 
proposed indicator $\eta^{\mbox{\scriptsize{H,En}}}_{\omega_i}$ and the exact indicator $\eta^{\mbox{\scriptsize{H,Ex}}}_{\omega_i}$.
In Figures \ref{fig:Spectral_ErrorDistri_Cross_larger} and \ref{fig:Spectral_ErrorDistri_Cross_smaller},
the energy error distributions for the last offline spaces and an intermediate offline space
obtained by the proposed indicator are shown respectively. 
We see clearly that how the energy error is reduced by enriching the space 
from an intermediate step to the final step. 
A similar situation is also seen from Figures \ref{fig:Spectral_ErrorDistri_Cross_U_larger} and \ref{fig:Spectral_ErrorDistri_Cross_U_smaller}
for the case with the exact indicator.

\subsection{Numerical results with the $L^2$ indicator}
\label{sec:623}
In this section, we present some numerical simulations to test the performance of the $L^2$ indicator.
We note that this is the most natural error indicator, as it is more efficient to compute and is
widely used for classical adaptive finite element methods \cite{AdaptiveFEM}.
However, this indicator does not work well for high contrast coefficients. 
In the simulation, we will conduct the same test as in Section \ref{sec:621}
with the indicator replaced by $\eta^{\mbox{\scriptsize{H,L2}}}_{\omega_i}$. 

%
\begin{table}[htb]
\centering
\begin{tabular}{|c|c|c|c|c|c|c|}
\hline 
\multirow{2}{*}{$\text{dim}(V_{\text{off}})$} &
\multicolumn{2}{c|}{  $\|u-u^{\text{off}} \|$ (\%) } &
\multicolumn{2}{c|}{  $\|u^{\text{snap}}-u^{\text{off}} \|$ (\%) }  \\
\cline{2-5} {}&
$\hspace*{0.8cm}   L^{2}_\kappa(D)   \hspace*{0.8cm}$ &
$\hspace*{0.8cm}   H^{1}_\kappa(D)  \hspace*{0.8cm}$&
$\hspace*{0.8cm}   L^{2}_\kappa(D)   \hspace*{0.8cm}$ &
$\hspace*{0.8cm}   H^{1}_\kappa(D)  \hspace*{0.8cm}$
\\
\hline\hline
       $1524$       &  $4.50$    & $31.29$ &  $4.49$    & $31.14$  \\
\hline
      $1913$      & $3.59$    & $26.88$ &  $3.57$    & $26.69$\\
\hline
      $2513$      & $2.43$    & $21.46$ &  $2.43$    & $20.89$\\
\hline
      $3006$      & $1.11$    & $17.11$ &  $1.12$    & $16.83$\\
\hline
       $4509$    & $0.06$  &$7.97$  &  $0.04$    & $7.37$\\
\hline
\end{tabular}
\caption{Convergence history for harmonic basis using the $L^2$ indicator with $\theta=0.7$ and $94$ iterations.
The snapshot space has dimension $7300$ giving $0.05\%$ and $3.02\%$ weighted $L^2$ and weighted energy errors. 
When using $12$ basis per interior coarse node, the weighted $L^2$ and the weighted energy errors will be $2.34\%$ and $19.77\%$, respectively, and the dimension of offline space is $4412$.
}
\label{table:HCC_theta.7_Harmonic_nv}
\end{table}
In Table \ref{table:HCC_theta.7_Harmonic_nv},
we present the convergence history of the adaptive enrichment algorithm
for $\theta=0.7$, and 
we observe a clear convergence of the algorithm.
Notice that, the algorithm requires $94$ iterations to achieve the desired accuracy. 
The dimension of the corresponding offline space is $4509$.
If we compare these results to the case with the proposed indicator,
we see that the $L^2$ indicator
will give a much larger offline space and a larger number of iterations,
in order to achieve a similar accuracy. 


\begin{figure}[htb]
 \centering
 \subfigure[Basis distribution for $L^2$ indicator]{\label{fig:HarmonicBasis_NoOfBasisPerCoarseNode_Cross_nv}
    \includegraphics[width = 0.30\textwidth, keepaspectratio = true]{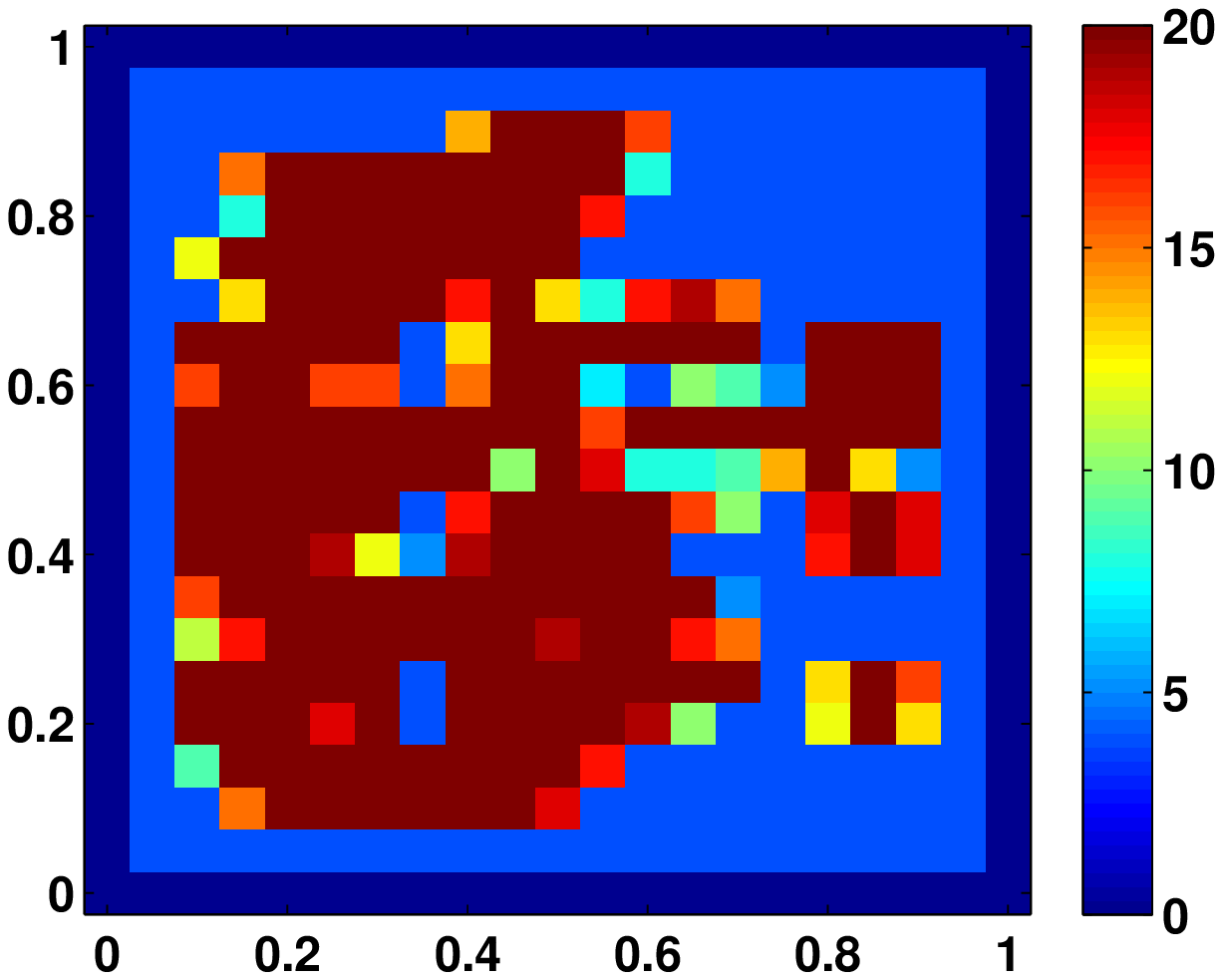}
   }
  \subfigure[Energy error with the last offline space]{\label{fig:Hm_ErrorDistri_Cross_tem_larger_nv}
    \includegraphics[width = 0.30\textwidth, keepaspectratio = true]{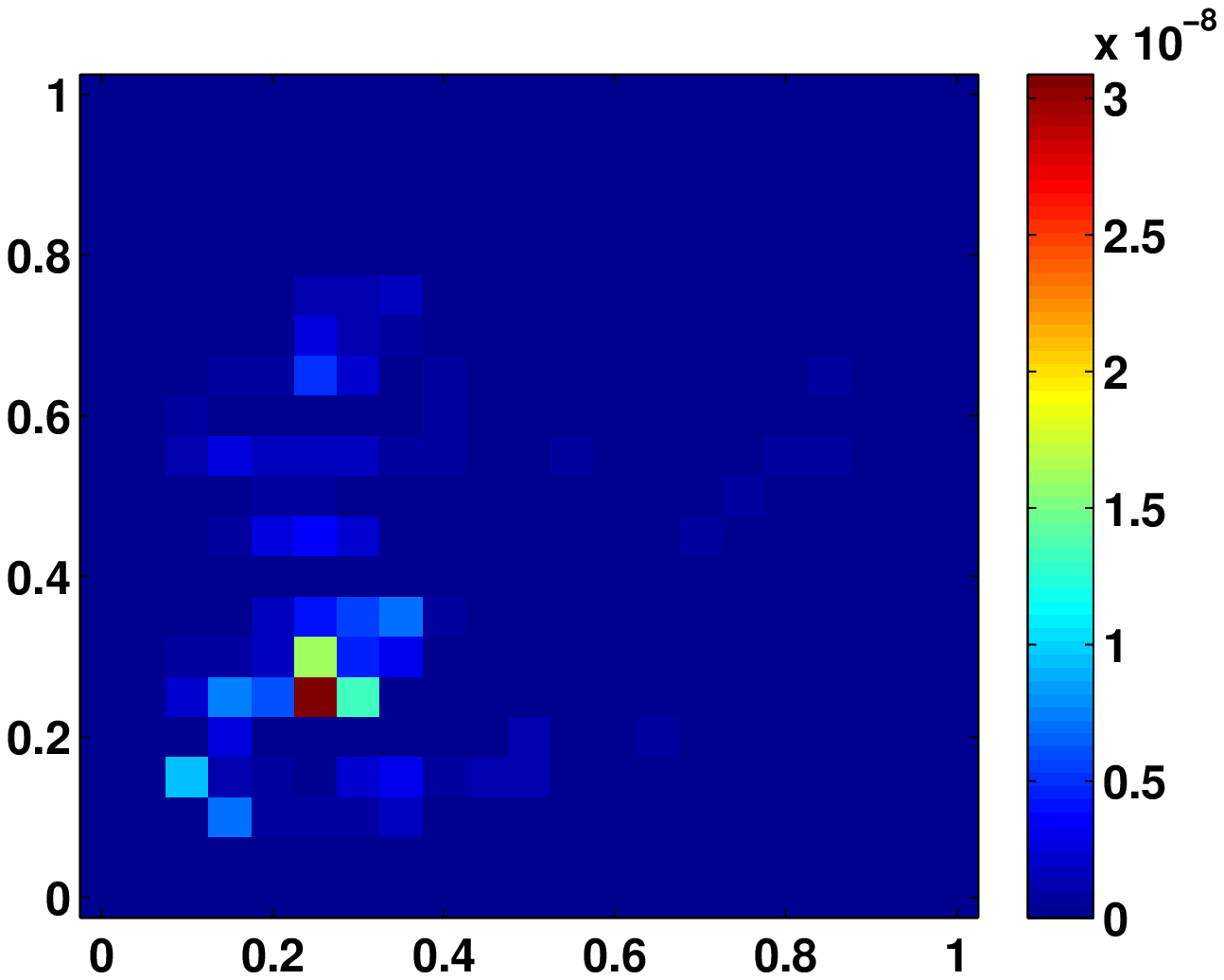}
   }
 \subfigure[Energy error with an intermediate offline space]{\label{fig:Hm_ErrorDistri_Cross_tem_smaller_nv}
    \includegraphics[width = 0.30\textwidth, keepaspectratio = true]{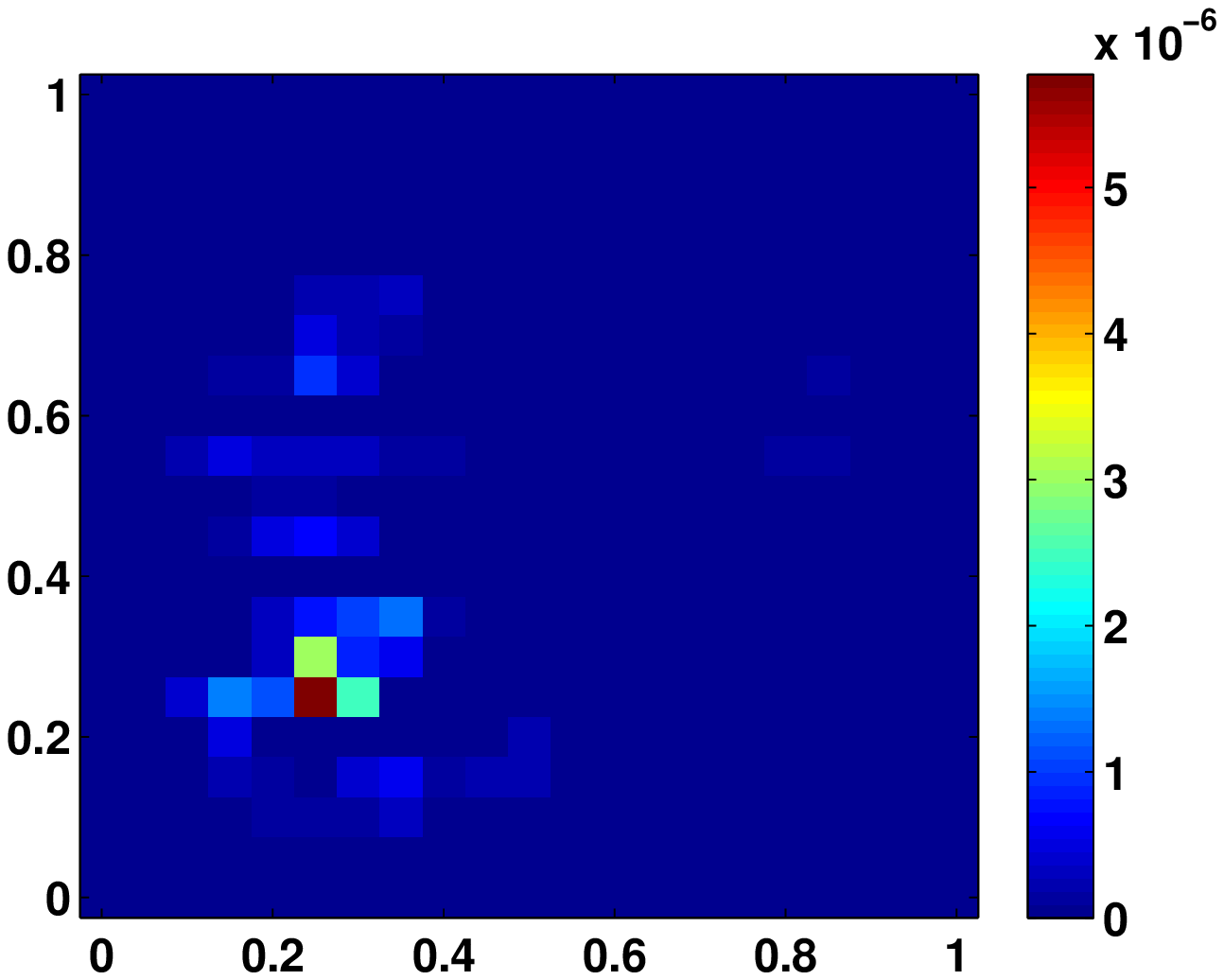}
   }
 \caption{Basis distribution and error distribution for harmonic basis with $L^2$ indicator.}
\label{fig:BasisPerNode_Harmonic_Cross_nv}
\end{figure}
%

Finally we will compare the basis function and error distributions
for the $L^2$ indicator with those for the proposed indicator.
In Figure \ref{fig:HarmonicBasis_NoOfBasisPerCoarseNode_Cross_nv},
the number of basis functions for each coarse node is shown. 
We observe that the distribution is similar to the case with the proposed indicator
shown in Figure \ref{fig:HarmonicBasis_NoOfBasisPerCoarseNode_Cross}.
We also observe that the number of basis functions for the $L^2$ indicator
is much larger than that for the proposed indicator. 
In Figures \ref{fig:Hm_ErrorDistri_Cross_tem_larger_nv} and \ref{fig:Hm_ErrorDistri_Cross_tem_smaller_nv},
the energy error distributions for the last offline spaces and an intermediate offline space
obtained by the $L^2$ indicator are shown respectively. 
We see clearly that how the energy error is reduced by enriching the space 
from an intermediate step to the final step. 
However, we also see a very slow decay in energy error
for the $L^2$ indicator. 


%
%
\subsection{Numerical results using snapshot solutions for the proposed indicator}
\label{sec:624}

In this section, we present numerical tests to show that 
our adaptive method is equally good when the proposed indicator $\eta^{\mbox{\scriptsize{H,En}}}_{\omega_i}$
is computed in the snapshot space. 
In particular, we will solve the local problems \eqref{eq:loc_Dirichlet} in the space of snapshots
instead of the fine scale space, in order to reduce the computational costs. 
We will again repeat the same test as in Section \ref{sec:621}. 
In Table \ref{table:HCC_theta.7_Harmonic_loc_snapshot} 
we present the convergence history of the adaptive enrichment algorithm
with $\theta=0.7$, and observe
a clear convergence of the algorithm.
Moreover, the algorithm requires $22$ iterations to achieve the desired accuracy. 
The dimension of the corresponding offline space is $3688$.
In addition, the error $u-u_{\text{off}}$ in relative weighted $L^2$ and energy norms
are $0.17\%$ and $7.83\%$ respectively, while the 
error $u_{\text{snap}}-u_{\text{off}}$ in relative weighted $L^2$ and energy norms
are $0.14\%$ and $7.26\%$ respectively.
If we compare these results with those for the proposed indicator (see Table \ref{table:Cross_theta.7_Harmonic}),
we see the use of snapshot space to compute the error indicator
will give a similar offline space and accuracy,
but with a larger number of iterations. 


%
\begin{table}[htb]
\centering
\begin{tabular}{|c|c|c|c|c|c|c|}
\hline 
\multirow{2}{*}{$\text{dim}(V_{\text{off}})$} &
\multicolumn{2}{c|}{  $\|u-u^{\text{off}} \|$ (\%) } &
\multicolumn{2}{c|}{  $\|u^{\text{snap}}-u^{\text{off}} \|$ (\%) }  \\
\cline{2-5} {}&
$\hspace*{0.8cm}   L^{2}_\kappa(D)   \hspace*{0.8cm}$ &
$\hspace*{0.8cm}   H^{1}_\kappa(D)  \hspace*{0.8cm}$&
$\hspace*{0.8cm}   L^{2}_\kappa(D)   \hspace*{0.8cm}$ &
$\hspace*{0.8cm}   H^{1}_\kappa(D)  \hspace*{0.8cm}$ 
\\
\hline\hline
       $1524$       &  $7.60$    & $50.86$ &  $7.59$    & $50.75$  \\
\hline
      $1772$      & $4.18$    & $27.08$ &  $4.18$    & $26.90$\\
\hline
      $2398$      & $2.41$    & $20.59$ &  $2.39$    & $20.36$\\
\hline

      $2824$      & $1.28$    & $13.98$ &  $1.25$    & $13.64$\\
\hline
       $3688$    & $0.17$  &$7.83$  &  $0.14$    & $7.26$\\
\hline
\end{tabular}
\caption{Convergence history for harmonic basis using snapshot space to compute the proposed indicator.
We take $\theta=0.7$ and the algorithm converges in $22$ iterations.
}
\label{table:HCC_theta.7_Harmonic_loc_snapshot}
\end{table}

\section{Convergence analysis}
\label{sec:analysis}

In this section, we will give the proofs for
the a-posteriori error estimates \eqref{eq:res1}-\eqref{eq:res2}
and the convergence of the adaptive enrichment algorithm. 

For each $i=1,2,\cdots, N$, we let $P_i: V \rightarrow \text{span}\{ \psi_k^{\omega_i,\text{off}} \}$
be the projection defined by
\begin{equation*}
P_i v = \sum_{k=1}^{l_i} \Big( \int_{\omega_i} \widetilde{\kappa} v \psi_k^{\omega_i,\text{off}} \Big) \psi_k^{\omega_i,\text{off}}.
\end{equation*}
The projection $P_i$ has following stability bound
\begin{equation}
\| \chi_i (P_i v) \|_{V_i} \leq C_{\text{stab}}^{\omega_i} \| v\|_{V_i}
\label{eq:eigenstab}
\end{equation}
where the constant $C_{\text{stab}}^{\omega_i} = \max ( 1,   H^{-1} (\lambda_{l_i+1}^{\omega_i})^{-\frac{1}{2}} )$. 
Moreover the following
convergence result holds
\begin{equation}
\| \chi_i ( v - P_i v) \|_{V_i} \leq C_{\text{conv}}^{\omega_i} (\lambda_{l_i+1}^{\omega_i})^{-\frac{1}{2}} \| v\|_{V_i}
\label{eq:eigenbound}
\end{equation}
where $C_{\text{conv}}^{\omega_i}$ is a uniform constant. 
We also define the projection $\Pi: V \rightarrow V_{\text{off}}$ by $\Pi v = \sum_{i=1}^N \chi_i (P_i v)$. 
For the analysis below, we let $C_{\text{stab}} = \max_{1\leq i\leq N} C_{\text{stab}}^{\omega_i}$
and $C_{\text{conv}} = \max_{1\leq i\leq N} C_{\text{conv}}^{\omega_i}$.

\subsection{Proof of Theorem \ref{thm:post}}

Let $v\in V$ be an arbitrary function in the space $V$. Using \eqref{eq:fine}, we have
\begin{equation*}
a(u-u_{\text{ms}},v) = a(u,v) - a(u_{\text{ms}},v) = (f,v) - a(u_{\text{ms}},v).
\end{equation*}
Then
\begin{equation*}
a(u-u_{\text{ms}},v) =  (f,v) - a(u_{\text{ms}},v)
= (f,v-\Pi v) + (f,\Pi v) - a(u_{\text{ms}},\Pi v) - a(u_{\text{ms}},v-\Pi v).
\end{equation*}
Thus, using \eqref{eq:coarse}, we have
\begin{equation}
a(u-u_{\text{ms}},v) = (f,v-\Pi v)  - a(u_{\text{ms}},v-\Pi v).
\label{eq:err1}
\end{equation}
Writing (\ref{eq:err1}) as a sum over coarse regions,
\begin{equation}
a(u-u_{\text{ms}},v) = \sum_{i=1}^N \Big( \int_{\omega_i} f (v-P_i v) \chi_i 
- \int_{\omega_i} a\nabla u_{\text{ms}} \cdot \nabla ( (v-P_i v)\chi_i ) \Big).
\label{eq:err2}
\end{equation}

Using the definition of $Q_i$, we see that \eqref{eq:err2} can be written as
\begin{equation*}
a(u-u_{\text{ms}},v) = \sum_{i=1}^N Q_i (v - P_i v). 
\end{equation*}
Thus, we have
\begin{equation*}
a(u-u_{\text{ms}},v) \leq \sum_{i=1}^N \| Q_i \| \| v-P_i v\|_{L^2(\omega_i)}. 
\end{equation*}
Using the definition of $\widetilde{\kappa}_i$, we have
\begin{equation*}
a(u-u_{\text{ms}},v) \leq \sum_{i=1}^N    ( \widetilde{\kappa}_i)^{-\frac{1}{2}} \| Q_i \| \|  \widetilde{\kappa}^{\frac{1}{2}} (v-P_i v)\|_{L^2(\omega_i)}. 
\end{equation*}
Thus, by the definition of the eigenvalue problem \eqref{offeig},
\begin{equation*}
a(u-u_{\text{ms}},v) \leq \sum_{i=1}^N    ( \widetilde{\kappa}_i)^{-\frac{1}{2}} (\lambda_{l_i+1}^{\omega_i})^{-\frac{1}{2}} \| Q_i \| \| v\|_{V_i}
\end{equation*}
The inequality \eqref{eq:res1} is then followed by taking $v = u-u_{\text{ms}}$
and $\sum_{i=1}^N \| v\|^2_{V_i} \leq C \|v\|^2_V$.

Using the definition of $R_i$, we see that \eqref{eq:err2} can be written as
\begin{equation*}
a(u-u_{\text{ms}},v) = \sum_{i=1}^N R_i ( \chi_i(v - P_i v) ). 
\end{equation*}
Thus, we have
\begin{equation*}
a(u-u_{\text{ms}},v) \leq \sum_{i=1}^N \| R_i \|_{V_i^*}  \|\chi_i( v-P_i v)\|_{V_i}. 
\end{equation*}
Using \eqref{eq:eigenbound}, 
\begin{equation*}
a(u-u_{\text{ms}},v) \leq C_{\text{conv}} \sum_{i=1}^N \| R_i \|_{V_i^*}  (\lambda_{l_i+1}^{\omega_i})^{-\frac{1}{2}} \| v\|_{V_i}. 
\end{equation*}
The inequality \eqref{eq:res2} is then followed by taking $v = u-u_{\text{ms}}$
and $\sum_{i=1}^N \| v\|^2_{V_i} \leq C \|v\|^2_V$.

\subsection{Some auxiliary lemmas}

In this section, we will prove some auxiliary lemmas which are required for the proof of
the convergence of the adaptive enrichment algorithm
stated in Theorem \ref{thm:conv}. 
We use the notation $P_i^m$ 
to denote the projection operator $P_i$ at the enrichment level $m$. 
Specifically, we define
\begin{equation*}
P_i^m v = \sum_{k=1}^{l_i^m} \Big( \int_{\omega_i} \widetilde{\kappa} v \psi_k^{\omega_i,\text{off}} \Big) \psi_k^{\omega_i,\text{off}}.
\end{equation*}

In Theorem \ref{thm:post}, we see that $\|R_i\|_{V_i^*}$
gives an upper bound of the energy error $\|u-u_{\text{ms}}\|_V$. 
We will first show that, $\|R_i\|_{V_i^*}$
is also a lower bound up to a correction term. 
To state this precisely, we define
\begin{equation}
S_m(\omega_i) = (\lambda_{l^m_i+1}^{\omega_i})^{-\frac{1}{2}} \sup_{v\in V_i} \frac{ |R_i(v - (P_i^{m+1} v)\chi_i) |}{\|v\|_{V_i}}
\label{eq:defSR}
\end{equation}
which is a measure on how small $(v-\chi_i P_i^{m+1}v)$ is. 
Notice that the residual $R_i$ is computed using the solution $u_{\text{ms}}^m$ obtained at enrichment level $m$.
We omit the index $m$ in $R_i$ to simplify notations. Next,
we will prove the following lemma.

\begin{lemma}
\label{lem:R1}
We have
\begin{equation}
\| R_i \|^2_{V_i^*} (\lambda_{l^m_i+1}^{\omega_i})^{-1} 
\leq 2(C_{\text{stab}}^{\omega_i})^2 (\lambda_{l^m_i+1}^{\omega_i})^{-1}   \|u_{\text{ms}}^{m+1}-u_{\text{ms}}^m\|_{V_i}^2 + 2S_m(\omega_i)^2
\label{eq:Rbound}
\end{equation}
\end{lemma}

\begin{proof}
By linearity
\begin{equation*}
R_i(v) = R_i( \chi_i (P_i^{m+1} v) ) + R_i(v - \chi_i(P_i^{m+1} v)).
\end{equation*}
Since $\chi_i(P_i^{m+1} v)$ is a test function in the space $V_{\text{off}}^{m+1}$, by the definition of $R_i$ and \eqref{eq:solve}, we have
\begin{equation*}
\begin{split}
R_i(\chi_i(P_i^{m+1} v)) &= \int_{\omega_i} f (P_i^{m+1} v) \chi_i - \int_{\omega_i} a\nabla u^m_{\text{ms}}\cdot \nabla ( (P_i^{m+1} v)\chi_i) \\
&= \int_{\omega_i} a\nabla u^{m+1}_{\text{ms}} \cdot \nabla ( (P_i^{m+1} v)\chi_i) - \int_{\omega_i} a\nabla u^m_{\text{ms}}\cdot \nabla ( (P_i^{m+1}v) \chi_i).
\end{split}
\end{equation*}
Using the stability estimate \eqref{eq:eigenstab},
\begin{equation*}
R_i( \chi_i(P_i^{m+1} v )) \leq \| u_{\text{ms}}^{m+1} - u_{\text{ms}}^m \|_{V_i} \| (P_i^{m+1} v)\chi_i \|_{V_i} 
\leq  C_{\text{stab}}^{\omega_i} \| u_{\text{ms}}^{m+1} - u_{\text{ms}}^m \|_{V_i}  \|v\|_{V_i}.
\end{equation*}
Thus, we obtain
\begin{equation}
\| R_i\|_{V_i^*} \leq 
C_{\text{stab}}^{\omega_i}  \|u_{\text{ms}}^{m+1}-u_{\text{ms}}^m\|_{V_i} + \sup_{v\in V_i} \frac{ |R_i(v - (P_i^{m+1} v)\chi_i) |}{\|v\|_{V_i}}
\label{eq:Rbound1}
\end{equation}
The inequality \eqref{eq:Rbound} follows from the definition of $S_m(\omega_i)$. 
\end{proof}

We remark that one can replace $u_{\text{ms}}^{m+1}$
by $u_{\text{snap}}$
and $P^{m+1}_i$ by $P^{\text{snap}}_i$,
where $P^{\text{snap}}_i$ is the projection onto the snapshot space defined by
\begin{equation*}
P_i^{\text{snap}} v = \sum_{k=1}^{W_i} \Big( \int_{\omega_i} \widetilde{\kappa} v \psi_k^{\omega_i,\text{off}} \Big) \psi_k^{\omega_i,\text{off}}.
\end{equation*}
We also define $S(\omega_i)$ by
\begin{equation*}
S(\omega_i) = (\lambda_{l^m_i+1}^{\omega_i})^{-\frac{1}{2}} \sup_{v\in V_i} \frac{ |R_i(v - (P_i^{\text{snap}} v)\chi_i) |}{\|v\|_{V_i}}.
\end{equation*}
Following the proof of the above lemma, we get
\begin{equation*}
\| R_i \|^2_{V_i^*} (\lambda_{l^m_i+1}^{\omega_i})^{-1} 
\leq 2 (C_{\text{stab}}^{\omega_i})^2 (\lambda_{l^m_i+1}^{\omega_i})^{-1}   \|u_{\text{snap}}-u_{\text{ms}}^m\|_{V_i}^2 + 2 S(\omega_i)^2
\end{equation*}
which suggests that $\| R_i \|^2_{V_i^*} (\lambda_{l^m_i+1}^{\omega_i})^{-1}$ gives a lower bound of
the error $\|u_{\text{snap}}-u_{\text{ms}}^m\|_{V_i}^2$
up to a correction term $S(\omega_i)^2$.

To prove Theorem \ref{thm:conv}, we will
need the following recursive properties for $S_m(\omega_i)$. 

\begin{lemma}
\label{lem:recurR}
For any $\alpha_R >0$, we have
\begin{equation}
S_{m+1}(\omega_i)^2 \leq (1+\alpha_R) C_R \frac{\lambda_{l_i^m+1}^{\omega_i}}{\lambda_{l_i^{m+1}+1}^{\omega_i}} S_m(\omega_i)^2
+ (1+\alpha_R^{-1}) D_R   (\lambda_{l_i^{m+1}+1}^{\omega_i})^{-1}  \|u_\text{ms}^{m+1}-u_{\text{ms}}^m\|_{V_i}^2
\label{eq:Sbound}
\end{equation}
where the enrichment level dependent constants $C_R$ and $D_R$ are defined by
\begin{equation*}
C_R = (1 + 2 C_{\text{conv}}^{\omega_i} (\lambda_{l_i^{m}+1}^{\omega_i})^{-\frac{1}{2}} (\lambda_{l_i^{m+1}+1}^{\omega_i})^{-\frac{1}{2}})^2
\quad\text{and}\quad
D_R = (C_{\text{stab}}^{\omega_i})^2  (1 +  2 C_{\text{conv}}^{\omega_i} (\lambda_{l_i^{m+1}+1}^{\omega_i})^{-\frac{1}{2}})^2.
\end{equation*}
\end{lemma}

\begin{proof}
By direct calculation, we have
\begin{equation}
\begin{split}
&\: \int_{\omega_i} f (v - (P_i^{m+2} v) \chi_i)  - \int_{\omega_i} a\nabla u_{\text{ms}}^{m+1} \cdot \nabla ( v- (P_i^{m+2}v)\chi_i ) \\
= &\: \int_{\omega_i} f (v - (P_i^{m+1} v )\chi_i)  - \int_{\omega_i} a\nabla u_{\text{ms}}^{m} \cdot \nabla ( v- (P_i^{m+1}v)\chi_i ) \\
&\: - \int_{\omega_i} a \nabla (u_{\text{ms}}^{m+1}-u_{\text{ms}}^m) \cdot \nabla (v - (P_i^{m+2}v)\chi_i) \\
&\: + \int_{\omega_i} f (P_i^{m+1}v - P_i^{m+2}v)\chi_i - \int_{\omega_i} a\nabla u_{\text{ms}}^m \cdot \nabla ( (P_i^{m+1}v-P_i^{m+2}v)\chi_i).
\end{split}
\label{err}
\end{equation}
By definition of $S_m(\omega_i)$, we have
\begin{equation}
S_m(\omega_i) = (\lambda_{l_i^m+1}^{\omega_i})^{-\frac{1}{2}}
\sup_{v\in V_i} \frac{ | \int_{\omega_i} f (v - (P_i^{m+1} v )\chi_i)  - \int_{\omega_i} a\nabla u_{\text{ms}}^{m} \cdot \nabla ( v- (P_i^{m+1}v)\chi_i ) | }{\| v\|_{V_i}}.
\end{equation}
Multiplying \eqref{err} by $(\lambda_{l_i^{m+1}+1}^{\omega_i})^{-\frac{1}{2}} \|v\|_{V_i}^{-1}$ and taking supremum with respect to $v$, we have
\begin{equation}
S_{m+1}(\omega_i) \leq (\frac{\lambda_{l_i^m+1}^{\omega_i}}{\lambda_{l_i^{m+1}+1}^{\omega_i}})^{\frac{1}{2}} S_m(\omega_i)
+ I_1 + I_2
\label{err1}
\end{equation}
where
\begin{equation*}
I_1 = (\lambda_{l_i^{m+1}+1}^{\omega_i})^{-\frac{1}{2}} \sup_{v\in V_i} \frac{ |- \int_{\omega_i} a \nabla (u_{\text{ms}}^{m+1}-u_{\text{ms}}^m) \cdot \nabla (v - (P_i^{m+2}v)\chi_i)| }{\|v\|_{V_i}}
\end{equation*}
and
\begin{equation*}
I_2 = (\lambda_{l_i^{m+1}+1}^{\omega_i})^{-\frac{1}{2}} \sup_{v\in V_i} \frac{ | \int_{\omega_i} f (P_i^{m+1}v - P_i^{m+2}v)\chi_i - \int_{\omega_i} a\nabla u_{\text{ms}}^m \cdot \nabla ( (P_i^{m+1}v-P_i^{m+2}v)\chi_i) |}{\|v\|_{V_i}}.
\end{equation*}
To estimate $I_1$, we use the stability estimate \eqref{eq:eigenstab} to obtain
\begin{equation*}
\begin{split}
&\int_{\omega_i} a \nabla (u_{\text{ms}}^{m+1}-u_{\text{ms}}^m) \cdot \nabla (v - (P_i^{m+2}v)\chi_i) \\
=& \int_{\omega_i} a \nabla (u_{\text{ms}}^{m+1}-u_{\text{ms}}^m) \cdot \nabla v
- \int_{\omega_i} a \nabla (u_{\text{ms}}^{m+1}-u_{\text{ms}}^m) \cdot \nabla ( (P_i^{m+2}v)\chi_i )\\
\leq &C_{\text{stab}}^{\omega_i} \| u_{\text{ms}}^{m+1} - u_{\text{ms}}^m\|_{V_i}  \|v\|_{V_i}
\end{split}
\end{equation*}
which implies
\begin{equation*}
I_1 \leq C_{\text{stab}}^{\omega_i} (\lambda_{l_i^{m+1}+1}^{\omega_i})^{-\frac{1}{2}} \| u_{\text{ms}}^{m+1} - u_{\text{ms}}^m \|_{V_i}. 
\end{equation*}
To estimate $I_2$, we use the definition of $R_i$ to obtain
\begin{equation*}
I_2 \leq (\lambda_{l_i^{m+1}+1}^{\omega_i})^{-\frac{1}{2}} \| R_i \|_{V_i^*}  \sup_{v\in V_i} \frac{ \| \chi_i(P_i^{m+1}v - P_i^{m+2}v) \|_{V_i}}{ \|v\|_{V_i}}.
\end{equation*}
By the convergence bound \eqref{eq:eigenbound} and the fact that
$\lambda_{l_i^{m+1}+1}^{\omega_i} < \lambda_{l_i^{m+2}+1}^{\omega_i}$, we have
\begin{equation*}
\| \chi_i(P_i^{m+1}v - P_i^{m+2}v) \|_{V_i} 
\leq  \| \chi_i(P_i^{m+1}v - v) \|_{V_i} +  \| \chi_i(v - P_i^{m+2}v) \|_{V_i} \leq 2 C_{\text{conv}}^{\omega_i}  (\lambda_{l_i^{m+1}+1}^{\omega_i})^{-\frac{1}{2}}  \|v\||_{V_i}
\end{equation*}
which implies
\begin{equation*}
I_2 \leq  2 C_{\text{conv}}^{\omega_i} (\lambda_{l_i^{m+1}+1}^{\omega_i})^{-1} \|R_i\|_{V_i^*}. 
\end{equation*}
Combining results and using \eqref{err1}, we get 
\begin{equation*}
S_{m+1}(\omega_i) \leq 
(\frac{\lambda_{l_i^m+1}^{\omega_i}}{\lambda_{l_i^{m+1}+1}^{\omega_i}})^{\frac{1}{2}} S_m(\omega_i)
+ C_{\text{stab}}^{\omega_i} (\lambda_{l_i^{m+1}+1}^{\omega_i})^{-\frac{1}{2}} \| u_{\text{ms}}^{m+1} - u_{\text{ms}}^m \|_{V_i}
+  2 C_{\text{conv}}^{\omega_i} (\lambda_{l_i^{m+1}+1}^{\omega_i})^{-1} \|R_i\|_{V_i^*}.
\end{equation*}
Using \eqref{eq:Rbound1} and the definition of $S_m(\omega_i)$,
\begin{equation*}
\begin{split}
S_{m+1}(\omega_i) \leq 
& (1 + 2 C_{\text{conv}}^{\omega_i} (\lambda_{l_i^{m}+1}^{\omega_i})^{-\frac{1}{2}} (\lambda_{l_i^{m+1}+1}^{\omega_i})^{-\frac{1}{2}}) (\frac{\lambda_{l_i^m+1}^{\omega_i}}{\lambda_{l_i^{m+1}+1}^{\omega_i}})^{\frac{1}{2}} S_m(\omega_i) \\
&+ C_{\text{stab}}^{\omega_i} (\lambda_{l_i^{m+1}+1}^{\omega_i})^{-\frac{1}{2}} (1 + 2 C_{\text{conv}}^{\omega_i} (\lambda_{l_i^{m+1}+1}^{\omega_i})^{-\frac{1}{2}})
\| u_{\text{ms}}^{m+1} - u_{\text{ms}}^m \|_{V_i}.
\end{split}
\end{equation*}
Hence, \eqref{eq:Sbound} is proved.
\end{proof}

Next, we consider the $L^2$-based residual $Q_i$
and prove similar inequalities. 
We define
\begin{equation}
S_m(\omega_i) = (\widetilde{\kappa}_i \lambda_{l^m_i+1}^{\omega_i})^{-\frac{1}{2}} \sup_{v\in L^2(\omega_i)} \frac{ |Q_i( v - P_i^{m+1} v) |}{\|v\|_{L^2(\omega_i)}}
\label{eq:defSQ}
\end{equation}
which is a measure on how small $(v- P_i^{m+1}v)$ is. 
Notice that the residual $Q_i$ is computed using the solution $u_{\text{ms}}^m$ obtained at enrichment level $m$.
We omit the index $m$ in $Q_i$ to simplify notations. 
We also note that we have used the same notation $S_m(\omega_i)$ as the case for the $H^{-1}$-based residual
to again simplify notations.
It will be clear which residual we are referring to when the notation $S_m(\omega_i)$ appears in the text. 
We define the jump of the coefficient in each coarse region by
\begin{equation*}
\beta_i = \frac{ \max_{x\in\omega_i} \kappa(x)}{ \min_{x\in\omega_i} \kappa(x) }.
\end{equation*}
Next,
we will prove the following lemma.

\begin{lemma}
\label{lem:Q1}
We have
\begin{equation}
\| Q_i \|^2 ( \widetilde{\kappa}_i \lambda_{l^m_i+1}^{\omega_i})^{-1} 
\leq 2 (C_{\text{inv}} \beta_i^{\frac{1}{2}} h^{-1})^2 ( \lambda_{l^m_i+1}^{\omega_i})^{-1}   \|u_{\text{ms}}^{m+1}-u_{\text{ms}}^m\|_{V_i}^2 + 2 S_m(\omega_i)^2
\label{eq:Qbound}
\end{equation}
\end{lemma}

\begin{proof}
By linearity
\begin{equation*}
Q_i(v) = Q_i( P_i^{m+1} v ) + Q_i(v - P_i^{m+1} v).
\end{equation*}
By the definition of $Q_i$ and \eqref{eq:solve}, we have
\begin{equation*}
\begin{split}
Q_i( P_i^{m+1} v) &= \int_{\omega_i} f (P_i^{m+1} v) \chi_i - \int_{\omega_i} a\nabla u^m_{\text{ms}}\cdot \nabla ( (P_i^{m+1} v)\chi_i) \\
&= \int_{\omega_i} a\nabla u^{m+1}_{\text{ms}} \cdot \nabla ( (P_i^{m+1} v)\chi_i) - \int_{\omega_i} a\nabla u^m_{\text{ms}}\cdot \nabla ( (P_i^{m+1}v) \chi_i)
\end{split}
\end{equation*}
which implies
\begin{equation*}
Q_i( P_i^{m+1} v ) \leq \| u_{\text{ms}}^{m+1} - u_{\text{ms}}^m \|_{V_i} \| (P_i^{m+1} v)\chi_i \|_{V_i}.
\end{equation*}
Using the inverse inequality,
\begin{equation*}
\| (P_i^{m+1} v)\chi_i \|_{V_i} \leq C_{\text{inv}} h^{-1} \| \widetilde{\kappa}^{\frac{1}{2}} P_i^{m+1} v\|_{L^2(\omega_i)}
\leq C_{\text{inv}} h^{-1} \| \widetilde{\kappa}^{\frac{1}{2}}  v\|_{L^2(\omega_i)}
\end{equation*}
where $C_{\text{inv}}$ is independent of the mesh size. 
Thus, we obtain
\begin{equation}
 (\widetilde{\kappa}_i)^{-\frac{1}{2}} \| Q_i\|_{V_i^*} \leq 
C_{\text{inv}} \beta_i^{\frac{1}{2}} h^{-1}  \|u_{\text{ms}}^{m+1}-u_{\text{ms}}^m\|_{V_i} 
+ (\widetilde{\kappa}_i)^{-\frac{1}{2}} \sup_{v\in L^2(\omega_i)} \frac{ |Q_i(v - P_i^{m+1} v) |}{\|v\|_{L^2(\omega_i)}}.
\label{eq:Qbound1}
\end{equation}
The inequality \eqref{eq:Qbound} follows from the definition of $S_m(\omega_i)$. 
\end{proof}

Next we will prove the following recursive property for $S_m(\omega_i)$. 
The proof follows from the same lines as Lemma \ref{lem:recurR}.

\begin{lemma}
\label{lem:recurQ}
For any $\alpha_Q >0$, we have
\begin{equation}
S_{m+1}(\omega_i)^2 \leq (1+\alpha_Q) C_Q \frac{\lambda_{l_i^m+1}^{\omega_i}}{\lambda_{l_i^{m+1}+1}^{\omega_i}} S_m(\omega_i)^2
+ (1+\alpha_Q^{-1}) D_Q   (\lambda_{l_i^{m+1}+1}^{\omega_i})^{-1}  \|u_\text{ms}^{m+1}-u_{\text{ms}}^m\|_{V_i}^2
\label{eq:SQbound}
\end{equation}
where the enrichment level dependent constants $C_R$ and $D_R$ are defined by
\begin{equation*}
C_Q = (1 + \beta_i^{\frac{1}{2}})^2
\quad\text{and}\quad
D_Q = C_{\text{inv}} \beta_i^{\frac{1}{2}} h^{-1} (2\widetilde{\kappa}_i + \beta_i^{\frac{1}{2}} ).
\end{equation*}
\end{lemma}

\begin{proof}
By direct calculation, we have
\begin{equation}
\begin{split}
&\: \int_{\omega_i} f (v - P_i^{m+2} v) \chi_i  - \int_{\omega_i} a\nabla u_{\text{ms}}^{m+1} \cdot \nabla ( (v- P_i^{m+2}v)\chi_i ) \\
= &\: \int_{\omega_i} f (v - P_i^{m+1} v )\chi_i  - \int_{\omega_i} a\nabla u_{\text{ms}}^{m} \cdot \nabla ( (v- P_i^{m+1}v)\chi_i ) \\
&\: - \int_{\omega_i} a \nabla (u_{\text{ms}}^{m+1}-u_{\text{ms}}^m) \cdot \nabla ((v - P_i^{m+2}v)\chi_i) \\
&\: + \int_{\omega_i} f (P_i^{m+1}v - P_i^{m+2}v)\chi_i - \int_{\omega_i} a\nabla u_{\text{ms}}^m \cdot \nabla ( (P_i^{m+1}v-P_i^{m+2}v)\chi_i).
\end{split}
\label{errQ1}
\end{equation}
By definition of $S_m(\omega_i)$, we have
\begin{equation}
S_m(\omega_i) = ( \widetilde{\kappa}_i \lambda_{l_i^m+1}^{\omega_i})^{-\frac{1}{2}}
\sup_{v\in L^2(\omega_i)} \frac{ | \int_{\omega_i} f (v - P_i^{m+1} v )\chi_i  - \int_{\omega_i} a\nabla u_{\text{ms}}^{m} \cdot \nabla ( (v- P_i^{m+1}v)\chi_i ) | }{\| v\|_{L^2(\omega_i)}}.
\end{equation}
Multiplying \eqref{errQ1} by $( \widetilde{\kappa}_i \lambda_{l_i^{m+1}+1}^{\omega_i})^{-\frac{1}{2}} \|v\|_{L^2(\omega_i)}^{-1}$ and taking supremum with respect to $v$, we have
\begin{equation}
S_{m+1}(\omega_i) \leq (\frac{\lambda_{l_i^m+1}^{\omega_i}}{\lambda_{l_i^{m+1}+1}^{\omega_i}})^{\frac{1}{2}} S_m(\omega_i)
+ I_1 + I_2
\label{errQ2}
\end{equation}
where
\begin{equation*}
I_1 = ( \widetilde{\kappa}_i \lambda_{l_i^{m+1}+1}^{\omega_i})^{-\frac{1}{2}} \sup_{v\in L^2(\omega_i)} \frac{ |- \int_{\omega_i} a \nabla (u_{\text{ms}}^{m+1}-u_{\text{ms}}^m) \cdot \nabla ((v - P_i^{m+2}v)\chi_i)| }{\|v\|_{L^2(\omega_i)}}
\end{equation*}
and
\begin{equation*}
I_2 = ( \widetilde{\kappa} \lambda_{l_i^{m+1}+1}^{\omega_i})^{-\frac{1}{2}} \sup_{v\in L^2(\omega_i)} \frac{ | \int_{\omega_i} f (P_i^{m+1}v - P_i^{m+2}v)\chi_i - \int_{\omega_i} a\nabla u_{\text{ms}}^m \cdot \nabla ( (P_i^{m+1}v-P_i^{m+2}v)\chi_i) |}{\|v\|_{L^2(\omega_i)}}.
\end{equation*}
To estimate $I_1$, we use the inverse inequality to obtain
\begin{equation*}
\int_{\omega_i} a \nabla (u_{\text{ms}}^{m+1}-u_{\text{ms}}^m) \cdot \nabla (v - (P_i^{m+2}v)\chi_i) 
\leq 2 C_{\text{inv}} h^{-1}  \| u_{\text{ms}}^{m+1} - u_{\text{ms}}^m\|_{V_i}  \| \widetilde{\kappa}^{\frac{1}{2}} v\|_{L^2(\omega_i)}
\end{equation*}
which implies
\begin{equation*}
I_1 \leq 2 C_{\text{inv}} ( \widetilde{\kappa}_i \lambda_{l_i^{m+1}+1}^{\omega_i})^{-\frac{1}{2}} \beta_i^{\frac{1}{2}} h^{-1} \| u_{\text{ms}}^{m+1} - u_{\text{ms}}^m \|_{V_i}. 
\end{equation*}
To estimate $I_2$, we use the definition of $Q_i$ to obtain
\begin{equation*}
I_2 \leq (\widetilde{\kappa}_i \lambda_{l_i^{m+1}+1}^{\omega_i})^{-\frac{1}{2}} \| Q_i \|  \sup_{v\in L^2(\omega_i)} \frac{ \| P_i^{m+1}v - P_i^{m+2}v \|}{ \|v\|_{L^2(\omega_i)}}
\end{equation*}
which implies
\begin{equation*}
I_2 \leq (\widetilde{\kappa}_i \lambda_{l_i^{m+1}+1}^{\omega_i})^{-\frac{1}{2}}  \beta_i^{\frac{1}{2}} \| Q_i \|.
\end{equation*}
Combining results and using \eqref{errQ2}, we get 
\begin{equation*}
S_{m+1}(\omega_i) \leq 
(\frac{\lambda_{l_i^m+1}^{\omega_i}}{\lambda_{l_i^{m+1}+1}^{\omega_i}})^{\frac{1}{2}} S_m(\omega_i)
+ 2 C_{\text{inv}} ( \widetilde{\kappa}_i \lambda_{l_i^{m+1}+1}^{\omega_i})^{-\frac{1}{2}} \beta_i^{\frac{1}{2}} h^{-1} \| u_{\text{ms}}^{m+1} - u_{\text{ms}}^m \|_{V_i}
+  (\widetilde{\kappa}_i \lambda_{l_i^{m+1}+1}^{\omega_i})^{-\frac{1}{2}}  \beta_i^{\frac{1}{2}} \|Q_i\|.
\end{equation*}
Using \eqref{eq:Qbound1}, 
\begin{equation*}
S_{m+1}(\omega_i) \leq 
 (1 + \beta_i^{\frac{1}{2}}) (\frac{\lambda_{l_i^m+1}^{\omega_i}}{\lambda_{l_i^{m+1}+1}^{\omega_i}})^{\frac{1}{2}} S_m(\omega_i) 
+ C_{\text{inv}} ( \lambda_{l_i^{m+1}+1}^{\omega_i})^{-\frac{1}{2}}  \beta_i^{\frac{1}{2}} h^{-1} (2\widetilde{\kappa}_i + \beta_i^{\frac{1}{2}} )
\| u_{\text{ms}}^{m+1} - u_{\text{ms}}^m \|_{V_i} 
\end{equation*}
Hence, \eqref{eq:SQbound} is proved.
\end{proof}

\subsection{Proof of Theorem \ref{thm:conv}}

In this section, we prove the convergence of the adaptive enrichment algorithm. 
We will give a unified proof for both the $L^2$-based and $H^{-1}$-based residuals. 
First of all, we use $\eta_i$ as a unified notation for the residuals, namely,
\begin{equation*}
\eta^2_i =
\begin{cases} 
& \|Q_i\|^2  (\widetilde{\kappa}_i \lambda^{\omega_i}_{l^m_i+1})^{-1},\quad \text{ for } L^2\text{-based residual}, \\
& \|R_i\|^2_{V_i^*}  (\lambda^{\omega_i}_{l^m_i+1})^{-1},\quad \text{ for } H^{-1}\text{-based residual}.
\end{cases}
\end{equation*}
Then Lemma \ref{lem:R1} and Lemma \ref{lem:Q1} can be written as
\begin{equation}
\eta_i^2 
\leq B_i ( \lambda_{l^m_i+1}^{\omega_i})^{-1}   \|u_{\text{ms}}^{m+1}-u_{\text{ms}}^m\|_{V_i}^2 + 2 S_m(\omega_i)^2
\label{eq:Uibound}
\end{equation}
where the constant $B_i$ is given by
\begin{equation*}
B_i =
\begin{cases} 
& 2(C_{\text{inv}} \beta_i^{\frac{1}{2}} h^{-1})^2,\quad \text{ for } L^2\text{-based residual} \\
& 2 (C_{\text{stab}}^{\omega_i,m+1})^2,\quad \text{ for } H^{-1}\text{-based residual}
\end{cases}
\end{equation*}
We remark that the definitions of $S_m(\omega_i)$ are given in \eqref{eq:defSQ} and \eqref{eq:defSR}
for the $L^2$-based and $H^{-1}$-based residuals respectively. 
Moreover, Lemma \ref{lem:recurR} and Lemma \ref{lem:recurQ} can be unified as
\begin{equation}
S_{m+1}(\omega_i)^2 \leq (1+\alpha_S) C_S \frac{\lambda_{l_i^m+1}^{\omega_i}}{\lambda_{l_i^{m+1}+1}^{\omega_i}} S_m(\omega_i)^2
+ (1+\alpha_S^{-1}) D_S   (\lambda_{l_i^{m+1}+1}^{\omega_i})^{-1}  \|u_\text{ms}^{m+1}-u_{\text{ms}}^m\|_{V_i}^2
\label{eq:Uibound1}
\end{equation}
where $\alpha_S = \alpha_Q, C_S=C_Q$ and $D_S=D_Q$ for the $L^2$-based residual
while $\alpha_S = \alpha_R, C_S=C_R$ and $D_S=D_R$ for the $H^{-1}$-based residual.
Notice that $\alpha_S>0$ is a constant defined uniformly over coarse regions and is 
to be determined. 
The convergence proof is based on \eqref{eq:Uibound} and \eqref{eq:Uibound1}. 

Let $0 < \theta < 1$. We choose an index set $I$ so that
\begin{equation}
\theta^2 \sum_{i=1}^N \eta_i^2
\leq \sum_{i\in I} \eta_i^2. 
\label{eq:indicator}
\end{equation}
We also assume there is a real number $\gamma$ with 
$0 < \gamma < 1$ satisfies
\begin{equation}
\gamma^2 \sum_{i=1}^n S_m(\omega_i)^2 \leq \sum_{i\in I} S_m(\omega_i)^2.
\end{equation}
We will then add basis function for those $\omega_i$ with $i\in I$. 
Then, using Theorem \ref{thm:post} and \eqref{eq:indicator}, we have
\begin{equation*}
\theta^2 \| u-u_{\text{ms}}^m\|_V^2 \leq \theta^2 C_{\text{err}} \sum_{i=1}^N \eta_i^2
\leq C_{\text{err}} \sum_{i\in I} \eta_i^2. 
\end{equation*}
By (\ref{eq:Uibound}), 
\begin{equation*}
\theta^2 \| u-u_{\text{ms}}^m\|_V^2 \leq 2 C_{\text{err}} \sum_{i=1}^N S_m(\omega_i)^2 + L_1 \| u_H^{m+1}-u_H^m\|_V^2
\end{equation*}
where
\begin{equation}
L_1 =  C_{\text{err}} \max_{1\leq i\leq N} \Big( B_i (\lambda_{l_i^m+1}^{\omega_i})^{-1}  \Big).
\label{eq:assumeL}
\end{equation}
Note that, by Galerkin orthogonality, we have
\begin{equation*}
\| u_{\text{ms}}^{m+1}-u_{\text{ms}}^m\|_V^2 = \|u-u_{\text{ms}}^m\|_V^2 - \|u-u_{\text{ms}}^{m+1}\|_V^2.
\end{equation*}
So, we have
\begin{equation*}
\theta^2 \| u-u_{\text{ms}}^m\|_V^2 \leq 2 C_{\text{err}} \sum_{i=1}^N S_m(\omega_i)^2 + L_1 (\|u-u_H^m\|_V^2 - \|u-u_H^{m+1}\|_V^2)
\end{equation*}
which implies
\begin{equation}
\| u-u_{\text{ms}}^{m+1} \|_V^2 \leq ( 1-\frac{\theta^2}{L_1} ) \| u-u_{\text{ms}}^m\|_V^2 + \frac{2C_{\text{err}}}{L_1} \sum_{i=1}^N S_m(\omega_i)^2.
\label{eq:conv1}
\end{equation}

On the other hand,
\begin{equation*}
\sum_{i=1}^N S_{m+1}(\omega_i)^2
= \sum_{i\in I} S_{m+1}(\omega_i)^2 + \sum_{i\ne I} S_{m+1}(\omega_i)^2
\end{equation*}
By (\ref{eq:Uibound1}) and that $S_{m+1}(\omega_i) = S_m(\omega_i)$ for $i\ne I$, we have
\begin{equation*}
\sum_{i=1}^N S_{m+1}(\omega_i)^2 
\leq 
\sum_{i\in I} \Big( (1+\alpha_S) C_S \frac{\lambda_{l_i^m+1}^{\omega_i}}{\lambda_{l_i^{m+1}+1}^{\omega_i}} S_m(\omega_i)^2
+ (1+\alpha_S^{-1}) D_S   (\lambda_{l_i^{m+1}+1}^{\omega_i})^{-1}  \|u_\text{ms}^{m+1}-u_{\text{ms}}^m\|_{V_i}^2 \Big) + 
\sum_{i\ne I} S_{m}(\omega_i)^2.
\end{equation*}
We assume the enrichment is obtained so that
\begin{equation*}
\delta = C_S \max_{1\leq i\leq N} \frac{\lambda_{l_i^m+1}^{\omega_i}}{\lambda_{l_i^{m+1}+1}^{\omega_i}} < 1.
\end{equation*}
We then have
\begin{equation*}
\sum_{i=1}^N S_{m+1}(\omega_i)^2  
\leq (1+\alpha_S) \sum_{i=1}^N S_{m}(\omega_i)^2 - (1+\alpha_S) (1-\delta) \sum_{i\in I} S_m(\omega_i)^2 +  \delta L_2 \|u_{\text{ms}}^{m+1}-u_{\text{ms}}^m\|_V^2
\end{equation*}
where
\begin{equation}
L_2 =  (1+\alpha_S^{-1}) \max_{1\leq i\leq N} \Big( D_S C_S^{-1} (\lambda_{l_i^m+1}^{\omega_i})^{-1}  \Big).
\label{eq:assumeL2}
\end{equation}
By assumption on $\gamma$,
\begin{equation*}
\sum_{i=1}^N S_{m+1}(\omega_i)^2  
\leq (1+\alpha_S) \sum_{i=1}^N S_{m}(\omega_i)^2 - (1+\alpha_S) (1-\delta) \gamma^2 \sum_{i=1}^N S_m(\omega_i)^2 +  \delta L_2 \|u_{\text{ms}}^{m+1}-u_{\text{ms}}^m\|_V^2.
\end{equation*}
Let $\rho = (1+\alpha_S)( 1 - (1-\delta) \gamma^2)$. We choose $\alpha_S>0$ small so that
$0<\rho<1$. The above is then written as
\begin{equation}
\sum_{i=1}^N S_{m+1}(\omega_i)^2  \leq \rho \sum_{i=1}^N S_{m}(\omega_i)^2 
+  \delta L_2 ( \|u-u_{\text{ms}}^m\|_V^2 - \|u-u_{\text{ms}}^{m+1}\|_V^2 ).
\label{eq:conv2}
\end{equation}
Next, we take a constant $\tau$ so that 
\begin{equation*}
\tau > 0, \quad \frac{2C_{\text{err}}}{\tau L_1} + \rho < 1.
\end{equation*}
Finally, we combine \eqref{eq:conv1} and \eqref{eq:conv2} to obtain the following
\begin{equation*}
\begin{split}
&\: \| u - u_{\text{ms}}^{m+1}\|_V^2 + \tau \sum_{i=1}^N S_{m+1}(\omega_i)^2 \\
\leq &\: ( 1-\frac{\theta^2}{L_1} ) \| u-u_{\text{ms}}^m\|_V^2 + \frac{2C_{\text{err}}}{L_1} \sum_{i=1}^N S_m(\omega_i)^2
+ \tau \rho \sum_{i=1}^N S_{m}(\omega_i)^2  + \tau  \delta L_2 ( \|u-u_{\text{ms}}^m\|_V^2 - \|u-u_{\text{ms}}^{m+1}\|_V^2 ).
\end{split}
\end{equation*}
Rearranging the terms, we have
\begin{equation*}
(1+\tau  \delta L_2) \|u-u_{\text{ms}}^{m+1}\|_V^2 + \tau \sum_{i=1}^N S_{m+1}(\omega_i)^2
\leq ( 1-\frac{\theta^2}{L_1} + \tau \delta L_2 ) \| u-u_{\text{ms}}^m\|_V^2 + (\frac{2C_{\text{err}}}{L_1}+\tau\rho) \sum_{i=1}^N S_m(\omega_i)^2.
\end{equation*}
Hence we obtain
\begin{equation*}
\| u-u_{\text{ms}}^{m+1}\|_V^2 + \frac{\tau}{1+\tau \delta L_2} \sum_{i=1}^N S_{m+1}(\omega_i)^2
\leq ( 1- \frac{\theta^2}{L_1(1+\tau \delta L_2)} ) \|u-u_{\text{ms}}^m\|_V^2 + \frac{\tau}{1+\tau \delta L_2} (\frac{2C_{\text{err}}}{\tau L_1}+\rho) \sum_{i=1}^N S_{m}(\omega_i)^2.
\end{equation*}
Therefore, Theorem \ref{thm:conv} is proved.

\section{Conclusions}

In this paper, we derive an a-posteriori error indicator 
 for the Generalized 
Multiscale Finite Element Method (GMsFEM). 
In particular,
we study an adaptive spectral enrichment procedure and derive an 
 error indicator which gives an estimate of the local error 
over coarse grid regions.
We consider two kinds of error indicators where one is based on
the $L^2$-norm of the local residual
and the other is based on the weighted $H^{-1}$-norm of the local residual
where the weight is related to the coefficient of the elliptic equation.
We show that the use of weighted $H^{-1}$-norm residual 
gives a more robust error indicator
which works well for cases with high contrast multiscale problems. The convergence analysis of the method is given. Numerical results are presented that demonstrate the robustness
of the proposed error indicators. We
show the convergence of the proposed indicators and their similarities to
the ones when exact solution is used in the indicator.
We compare the performance of
the weighted $H^{-1}$-based indicator with
that of the $L^2$-based indicator for high-contrast
problems. Our numerical results show that the former is more appropriate
for high-contrast multiscale problems.

Although the results presented in this paper are encouraging, there
is scope for further exploration. 
As our intent here was to derive and demonstrate
the robustness of error indicators for challenging high-contrast 
multiscale problems, we did not consider
the fine-grid discretization error and assumed that the coarse-grid
error is the main contributor, and thus assuming that the fine-grid
solution is the desired quantity. In general when solving continuous
PDEs, one can also add fine-grid discretization errors due to basis
computations. This will be a subject of our future research.

\bibliographystyle{plain}
\bibliography{references}
\end{document}